\documentclass[journal,10pt]{IEEEtran}

\ifCLASSINFOpdf
\else
\fi

\usepackage{amsmath,amssymb}
\usepackage{xcolor}
\usepackage{graphicx}

\newtheorem{theorem}{Theorem}[section]
\newtheorem{corollary}[theorem]{Corollary}
\newtheorem{lemma}[theorem]{Lemma}
\newtheorem{remark}[theorem]{Remark}
\newtheorem{proposition}[theorem]{Proposition}
\newtheorem{definition}{Definition}[section]

\newcommand{\ba}{\begin{array}}
\newcommand{\ea}{\end{array}}

\newcommand{\Br}{\mathbb{R}}

\newcommand{\st}{\textnormal{s.t.}}
\newcommand{\tr}{\textnormal{tr}\,}

\newcommand{\rank}{\textnormal{rank}}

\newcommand{\re}{\mbox{\rm Re }}
\newcommand{\im}{\mbox{\rm Im }}
\newcommand{\CSI}{\mathbb S}
\newcommand{\F}{\mathcal{F}}
\newcommand{\ov}{\overline}

\newcommand{\SI}{\mathbb S}
\newcommand{\RR}{\mathbf R}
\newcommand{\C}{\mathbb C}

\newcommand{\ACal}{\mathcal{A}}

\newcommand{\XCal}{\mathcal{X}}

\newcommand{\linear}{\boldsymbol L}\,
\newcommand{\matr}{\boldsymbol M}\,
\,
\newcommand{\vect}{\boldsymbol V}\,

\newcommand{\etal}{{et al.\ }}

\include{defs}

\begin{document}

\title{Low-M-Rank Tensor Completion \\ and Robust Tensor PCA}

\author{Bo Jiang, Shiqian Ma, and Shuzhong Zhang
\thanks{B. Jiang is with the Research Institute for Interdisciplinary Sciences, School of Information Management and Engineering, Shanghai Univsersity of Finance and Economics, Shanghai 200433, China. Research of this author was supported in part by NSFC grants 11771269 and 11831002, and Program for Innovative Research Team of Shanghai University of Finance and Economics.}
\thanks{S. Ma is with the Department of Mathematics, University of California, Davis, CA 95616 USA. Research of this author was supported in part by a startup package in Department of Mathematics at UC Davis.}
\thanks{S. Zhang is with the Department of Industrial and Systems Engineering, University of Minnesota, and Institute of Data and Decision Analytics, The Chinese University of Hong Kong (Shenzhen). Research of this author was supported in part by the National Science Foundation (Grant CMMI-1462408), and in part by Shenzhen Fundamental Research Fund (Grant KQTD2015033114415450).}
\thanks{Manuscript received April xx, 2018; revised xxx xx, 2018.}}

\markboth{Manuscript,~Vol.~xx, No.~xx, xxx 2018}
{Shell \MakeLowercase{\textit{et al.}}: Low-M-Rank Tensor Completion and Robust Tensor PCA}

\maketitle

\begin{abstract}
In this paper, we propose a new approach to solve low-rank tensor completion and robust tensor PCA. Our approach is based on some novel notion of (even-order) tensor ranks, to be called the M-rank, the symmetric M-rank, and the strongly symmetric M-rank. We discuss the connections between these new tensor ranks and the CP-rank and the symmetric CP-rank of an even-order tensor.
We show that the M-rank provides a reliable and easy-computable approximation to the CP-rank. As a result, we propose to replace the CP-rank by the M-rank in the low-CP-rank tensor completion and robust tensor PCA. Numerical results suggest that our new approach based on the M-rank outperforms existing methods that are based on low-n-rank, t-SVD and KBR approaches for solving low-rank tensor completion and robust tensor PCA when the underlying tensor has low CP-rank.
\end{abstract}


\IEEEpeerreviewmaketitle

\section{Introduction}

\IEEEPARstart{T}{ensor} data have appeared frequently in applications such as computer vision \cite{WA04}, psychometrics \cite{Harshman-70,Carroll-Chang-70},
diffusion magnetic resonance imaging \cite{GTDPMR08,BV08,QYW10}, quantum entanglement problem \cite{HS10}, spectral hypergraph theory \cite{HQ12} and higher-order Markov chains \cite{LN11}.
Tensor-based multi-dimensional data analysis has shown that tensor models can take full advantage of the multi-dimensionality structures of the data, and generate more useful information.
A common observation for huge-scale data analysis is that the data exhibits a low-dimensionality property, or its most representative part lies in a low-dimensional subspace.
To take advantage of this low-dimensionality feature of the tensor data, it becomes imperative that the rank of a tensor, 
which is unfortunately a notoriously thorny issue, is well understood and computationally manageable.
The most commonly used definition of tensor rank is the so-called {\it CP-rank}, where ``C'' stands for CANDECOMP and ``P'' corresponds to PARAFAC and these are two alternate names for the
same tensor decomposition. 
The CP-rank is the most natural notion of tensor rank, which is also very 
difficult to compute numerically. 
\begin{definition}\label{def:asymmetric-cp-rank}
Given a $d$-th order tensor $\mathcal{F} \in \C^{n_1\times n_2\times \cdots \times n_{d}}$ in complex domain, its {\it CP-rank} (denoted as $\rank_{CP}(\mathcal {F})$) is the smallest integer $r$ exhibiting the following decomposition
\begin{equation}\label{formula:asymemtric-decomp}
\mathcal{F} = \sum_{i=1}^{r} a^{1,i}\otimes a^{2,i}\otimes \cdots \otimes a^{d,i},
\end{equation}
where $a^{k,i} \in \C^{n_i}$ for $k=1,\ldots, d$ and $i=1,\dots, r$ and $\otimes$ denotes outer product. Similarly, for a real-valued tensor $\mathcal{F} \in \Br^{n_1\times n_2\times \cdots \times n_{d}}$, its CP-rank in the real domain (denoted as $\rank_{CP}^{\Br}(\mathcal {F})$) is the smallest integer $r$ such that there exists a real-valued decomposition \eqref{formula:asymemtric-decomp} with $a^{k,i} \in \Br^{n_i}$ for $k=1,\ldots, d$ and $i=1,\dots, r$.
\end{definition}

An extreme case is $r=1$, where $\F$ is called a rank-$1$ tensor in this case.
For a given tensor, finding its best rank-one approximation, also known as finding the largest eigenvalue of a given tensor, has been studied in \cite{Lim05,Q05,DeLathauwer-2000-best-rank-1,KM11,Jiang-Ma-Zhang-TPCA-2013}.
It should be noted that the CP-rank of a real-valued tensor can be different over $\Br$ and $\C$; i.e., it may hold that $\rank_{CP}(\mathcal{F}) < \rank_{CP}^{\Br}(\mathcal{F})$
for a real-valued tensor $\mathcal{F}$. For instance, a real-valued $2 \times 2 \times 2$ tensor $\mathcal{F}$ is given in~\cite{Kruskal1983} and it can be shown that $\rank_{CP}^{\Br}(\mathcal {F})=3$ while 
$\rank_{CP}(\mathcal {F})=2$.
In this paper, we shall focus on the notion of CP-rank in the complex domain and discuss two low-CP-rank tensor recovery problems: tensor completion and robust tensor PCA.


Low-CP-rank tensor recovery problem seeks to recover a low-CP-rank tensor based on limited observations. This problem can be formulated as
\begin{equation}\label{prob:tensor-completion}
\min\limits_{\mathcal{X}\in\C^{n_1\times n_2\cdots \times n_{d}}} \ \rank_{CP}({\mathcal{X}}), \quad \st \ \linear(\mathcal{X}) = b,
\end{equation}
where $\linear: \C^{n_1\times n_2\cdots \times n_{d}} \to \C^p$ is a linear mapping and $b\in\C^p$ denotes the observations of $\mathcal{X}$ under $\linear$. Low-CP-rank tensor completion is a special case of \eqref{prob:tensor-completion} where the linear mapping in the constraint picks certain entries of the tensor. The low-CP-rank tensor completion can be formulated as
\begin{equation}\label{prob:tensor-completion-special}
\min\limits_{\mathcal{X}\in\C^{n_1\times n_2\cdots \times n_{d}}} \ \rank_{CP}({\mathcal{X}}), \quad \st \ P_\Omega(\mathcal{X}) = P_\Omega(\mathcal{X}_0),
\end{equation}
where $\mathcal{X}_0\in\C^{n_1\times n_2\cdots \times n_{d}}$, $\Omega$ is an index set and
\[
[P_\Omega(\mathcal{X})]_{i_1,i_2,\ldots,i_d} = \left\{ \ba{ll} \mathcal{X}_{i_1,i_2,\ldots,i_d}, & \mbox{ if } (i_1,i_2,\ldots,i_d) \in \Omega, \\
                                                                0                      , & \mbox{ otherwise.}  \ea
\right.
\]

In practice, the underlying low-CP-rank tensor data $\mathcal{F}\in\C^{n_1\times n_2\cdots \times n_{d}}$ may be heavily corrupted by a sparse noise tensor $\mathcal{Z}$. To identify and remove the noise, one can solve the following robust tensor PCA problem:
\begin{equation}\label{prob:robust-tensor-recovery}
\min\limits_{\mathcal{Y},\mathcal{Z}\in \C^{n_1\times n_2\cdots \times n_{d}}} \ \rank_{CP}({\mathcal{Y}}) + \lambda\|\mathcal{Z}\|_0, \quad \st \ \mathcal{Y} + \mathcal{Z} = \mathcal{F},
\end{equation}
where $\lambda>0$ is a weighting parameter and $\|\mathcal{Z}\|_0$ denotes the number of nonzero entries of $\mathcal{Z}$.

Solving \eqref{prob:tensor-completion} and \eqref{prob:robust-tensor-recovery}, however, are nearly impossible 
in practice. In fact, determining the CP-rank of a given tensor is known to be NP-hard in general \cite{H90}.
Worse than an average NP-hard problem, computing the CP-rank for small size instances remains a difficult task.
For example, a particular $9 \times 9 \times 9$ tensor is cited in \cite{K89} and its CP-rank is only known to be in between $18$ and $23$ to this date.
The above low-CP-rank tensor recovery problems \eqref{prob:tensor-completion} and \eqref{prob:robust-tensor-recovery}
thus appear to be quite hopeless. One way out of this dismay is to approximate the CP-rank by some more reasonable objects.
Since computing the rank of a matrix is easy, a classical way for tackling tensor optimization problem is
to unfold the tensor into certain matrix and then resort to some well established solution methods for matrix-rank optimization. A typical matrix unfolding technique is the so-called mode-$n$ matricization \cite{KB09}. Specifically, for a given tensor $\F\in \C^{n_1\times n_2\times \cdots\times n_d}$, we denote its mode-$n$ matricization by $F(n)$, which is obtained by arranging the $n$-th index of $\F$ as the column index of $F(n)$ and merging other indices of $\F$ as the row index of $F(n)$.
The Tucker rank (or the mode-$n$-rank) of $\F$ is defined as the vector $(\rank(F(1)),\ldots,\rank(F(d)))$. For simplicity, the averaged Tucker rank
is often used, to be denoted as
\[\rank_n(\F) := \frac{1}{d}\sum_{j=1}^d \rank(F(j)). \]

As such, $\rank_n(\F)$  is much easier to compute than the CP-rank, since it is just the average of $d$ matrix ranks. Therefore, the following low-n-rank minimization model has been proposed for tensor recovery \cite{LMWY09,Gandy-Recht-Yamada-2011}:
\begin{equation}\label{prob:low-n-rank-min}
\min\limits_{\mathcal{X}\in\C^{n_1\times n_2\cdots \times n_{d}}} \ \rank_{n}({\mathcal{X}}), \quad \st \ \linear(\mathcal{X}) = b,
\end{equation}
where $\linear$ and $b$ are the same as the ones in \eqref{prob:tensor-completion}. Since minimizing the rank function is still difficult, it was suggested in \cite{LMWY09} and \cite{Gandy-Recht-Yamada-2011} to convexify the matrix rank function by the nuclear norm, which has become a common practice due to the seminal works on matrix rank minimization (see, e.g., \cite{Fazel-thesis-2002,Recht-Fazel-Parrilo-2007,Candes-Recht-2008,Candes-Tao-2009}). That is, the following convex optimization problem is solved instead of \eqref{prob:low-n-rank-min}:
\begin{equation}\label{prob:low-n-rank-min-nuclear}
\min\limits_{\mathcal{X}\in\C^{n_1\times n_2\cdots \times n_{d}}} \ \frac{1}{d}\sum_{j=1}^d \|X(j)\|_*, \quad \st \ \linear(\mathcal{X}) = b,
\end{equation}
where the nuclear norm $\|X(j)\|_*$ is defined as the sum of singular values of matrix $X(j)$.

However, to the best of our knowledge, the relationship between the CP-rank and the Tucker rank of a tensor is still unclear so far,
although it is easy to see (from similar argument as in Theorem~\ref{theorem:asymmetric-bound}) that the averaged Tucker rank is a lower bound for the CP-rank. 
In fact, there is a substantial gap between the averaged Tucker rank and the CP-rank. 
In Proposition~\ref{tight-example}, we present an $n\times n \times n \times n$ tensor whose CP-rank is $n$ times the averaged Tucker rank.
Moreover, in the numerical experiments we found two types of tensors whose CP-rank is strictly larger than the averaged Tucker rank; see Table~\ref{CP-approximation-tab1} and Table~\ref{CP-approximation-tab2}.
The theoretical guarantee of model~\eqref{prob:low-n-rank-min-nuclear} has been established in~\cite{Tomioka-Suzuki-Hayashi-NIPS-2011}, which states that if the number of observations is at the order of $O(rn^{d-1})$, then with high probability the original tensor with Tucker rank $(r,r,\cdots,r)$ can be successfully recovered by solving \eqref{prob:low-n-rank-min-nuclear}.
Therefore, model~\eqref{prob:low-n-rank-min-nuclear} and its variants have become popular in the area of tensor completion.
However, our numerical results show that unless the CP-rank is extremely small, \eqref{prob:low-n-rank-min-nuclear} usually fails to recover the tensor; see Table~\ref{tensor-completion-tab1}.
As a different tensor unfolding technique, the square unfolding was proposed by Jiang \etal\cite{Jiang-Ma-Zhang-TPCA-2013} for the tensor rank-one approximation problem.
This technique was also considered by Mu \etal \cite{Mu-Huang-Wright-Goldfarb-2013} for tensor completion problem. 
Mu \etal \cite{Mu-Huang-Wright-Goldfarb-2013} showed that when the square unfolding is applied, the number of observations required to recover the tensor is of the order of $O(r^{\lfloor \frac{d}{2}\rfloor}n^{\lceil \frac{d}{2}\rceil})$, which is significantly less than that required by \eqref{prob:low-n-rank-min-nuclear}.
Recently, Yuan and Zhang \cite{YuanZhang2016} proposed to minimize the tensor nuclear norm for tensor completion and the sample size requirement for successful recovery can be further reduced to $O(r^{ \frac{1}{2}}n^{ \frac{d}{2}}\log(n))$. However, the corresponding optimization problem is computationally intractable as computing tensor nuclear norm itself is NP-hard \cite{FriedlandLim2018}. To alleviate this difficulty, Barak and Moitra \cite{BoazMoitra2016} proposed a polynomial time algorithm based on the sixth level of the sum-of-squares hierarchy and the required sample size is almost as low as the one for the exact tensor nuclear norm minimization. However, this algorithm is more of a theoretical interest as the size of resulting SDP is exceedingly large. For instance, to recover an $n \times n \times n$ tensor, the matrix variable could be of the size $n^3 \times n^3$, which is undoubtedly polynomial in $n$ but is already very challenging to solve when $n \ge 20$.
%


Note that in the definition of CP-rank, $\F\in\C^{n_1\times n_2\times\cdots\times n_d}$ is decomposed into the sum of asymmetric rank-one tensors. If $\F$ is a super-symmetric tensor, i.e., its component is invariant under any permutation of the indices, then a natural extension is to decompose $\F$ into the sum of symmetric rank-one tensors, and this leads to the definition of symmetric CP-rank (see, e.g., \cite{Comon2008}).
\begin{definition}\label{def:complex-sym-cp-rank}
Given a $d$-th order $n$-dimensional super-symmetric complex-valued tensor $\mathcal{F}$, its symmetric CP-rank (denoted by $\rank_{SCP}(\mathcal {F}) $) is the smallest integer $r$ such that
\begin{equation*}
\mathcal{F} = \sum_{i=1}^{r} \underbrace{a^i\otimes \cdots \otimes a^i}_{d},
\end{equation*}
with $a^i \in \C^n, i=1,\ldots,r$.
\end{definition}

It is obvious that 
$\rank_{CP}(\F) \leq \rank_{SCP}(\F)$ for any given super-symmetric tensor $\F$. In the matrix case, i.e., when $d=2$, it is known that the rank and symmetric rank are identical. {However, in the higher order case,
whether or not the CP-rank equals the symmetric CP-rank has remained an interesting and challenging open problem, known as Comon's conjecture \cite{Comon2008}. Very recently a counterexample disproving Comon's conjecture was claimed; see~\cite{Shitov2017}. Some earlier attempts to settle Comon's conjecture include \cite{Zhang2014,Friedland2016}.}

At this point, it is important to remark that the CP-rank stems from the idea of decomposing a general tensor into a sum of simpler -- viz.\ {\it rank-one}\/ in this context -- tensors. The nature of the ``simpler components'' in the sum, however, can be made flexible and inclusive. In fact, in many cases it does not have to be a rank-one tensor as in the CP-decomposition.
In particular, for a $2d$-th order tensor, the ``simple tensor'' being decomposed into could be the outer product of two tensors with lower degree, which is $d$ in this paper, and we call this new decomposition the {\it M-decomposition}. It is easy to see (will be discussed in more details later) that after square unfolding, each term in the M-decomposition is actually a rank-one matrix. Consequently, the minimum number of such simple terms can be regarded as a {\it rank}, or indeed the {\it M-rank} in our context, to be differentiated from other existing notions of tensor ranks.
By imposing further symmetry on the ``simple tensor'' that composes the M-decomposition,
the notion of symmetric M-decomposition (symmetric M-rank), and strongly symmetric M-decomposition (strongly symmetric M-rank) naturally follow. We will introduce the formal definitions later. The merits of the M-rank are twofold. First, for some structured tensors, we can show, through either theoretical analysis or numerical experimentations, that the M-rank is much better than the averaged Tucker rank in terms of approximating the CP-rank. Second, for low-CP-rank tensor recovery problems, the low-M-rank approach can improve the recoverability and our numerical tests suggest that the M-rank remain a good approximation to the CP-rank even in the presence of some gross errors.

{\subsection{Related Work}
	
For tensor completion problems, efficient algorithms such as alternating direction method of multipliers and Douglas-Rachford operator splitting methods were proposed in \cite{LMWY09} and \cite{Gandy-Recht-Yamada-2011} to solve \eqref{prob:low-n-rank-min-nuclear}.  Mu \etal \cite{Mu-Huang-Wright-Goldfarb-2013} suggested to minimize the nuclear norm of a matrix obtained by a square unfolding of the original tensor. The idea of square unfolding also appeared earlier in \cite{Jiang-Ma-Zhang-TPCA-2013} for tensor eigenvalue problem.
Recently, more sophisticated tensor norms are proposed (see \cite{YuanZhang2016,BoazMoitra2016}) to approximate the CP-rank in tensor completion problems. However, these norms are often computationally expensive or even intractable \cite{FriedlandLim2018}. Recently, matrix (tensor) completion algorithms have been used to find a background scene known as the background-initialization problem.
In \cite{SobralZahzah2017,SobralBouwmansZahzah2015}, many matrix completion and tensor completion algorithms are comparatively evaluated under the Scene Background Initialization data set. More recently, a spatiotemporal low-rank modeling method \cite{JavedMahmoodBouwmansJung2016} and SVD-based tensor-completion technique \cite{KajoKamelRuichekMalik2018} were further studied with the aim to enhance the performance of the background-initialization.

There exist comprehensive reviews for matrix robust PCA with a comparative evaluation in surveillance video data \cite{BouwmansZahzah2014,BouwmansEtal2017}. Convexifying the robust tensor PCA problem \eqref{prob:robust-tensor-recovery} was also studied by Tomioka \etal \cite{Tomioka-Hayashi-Kashima-2011} and Goldfarb and Qin \cite{Goldfarb-Qin-tensor-2013}. Specifically, they used the averaged Tucker rank to replace the CP-rank of $\mathcal{Y}$ and $\|\mathcal{Z}\|_1$ to replace $\|\mathcal{Z}\|_0$ in the objective of \eqref{prob:robust-tensor-recovery}. However, we observe from our numerical experiments (see Table~\ref{tensor-rpca-tab}) that this model cannot recover $\mathcal{Y}$ well when
$\mathcal{Y}$ is of low CP-rank. Other works on this topic include \cite{Signoretto-2011,Kressner-Steinlechner-Vandereycken-2013,Krishnamurthy-Singh-2013}. Specifically, \cite{Signoretto-2011} compared the performance of the convex relaxation of the low-n-rank minimization model and the low-rank matrix completion model on applications in spectral image reconstruction. \cite{Kressner-Steinlechner-Vandereycken-2013} proposed a Riemannian manifold optimization algorithm for finding a local optimum of the Tucker rank constrained optimization problem. \cite{Krishnamurthy-Singh-2013} studied some adaptive sampling algorithms for low-rank tensor completion.
There have been many recent works \cite{LiWangHuCai2016,ZhouZhangPeng2016,XiaSunChen2018,LuEtal2016,XieZhaoMengXu2017} on background subtraction of surveillance video via robust tensor PCA. In particular, an Outlier Iterative Removal algorithm \cite{LiWangHuCai2016} was proposed to reduce the video size and extract the background from the discriminative frame set. \cite{ZhouZhangPeng2016} used the low-rank approximation to exploit the inherent correlation of tensor data. In \cite{XiaSunChen2018}, an adaptive rank adjustment procedure was incorporated in the regularized tensor decomposition method to obtain accurate background component when the scene is changing at different time or places. Recently, t-SVD based tensor rank \cite{ZhangAeron2017,LuEtal2016} and KBR based tensor rank \cite{XieZhaoMengXu2017} have been proposed and applied in both tensor completion and robust tensor PCA. In \cite{ZhangAeron2017,LuEtal2016}, the tensor was transformed to the Fourier domain through the fast Fourier transform, and then the decomposition was conducted in the Fourier domain. The KBR based tensor rank \cite{XieZhaoMengXu2017} promoted the low-rankness of the target tensor via the sparsity of the associated core tensor. These two newly proposed tensor norms were well justified by the surveillance video examples (see the numerical results in \cite{ZhangAeron2017,LuEtal2016,XieZhaoMengXu2017} and also Figures \ref{CompletionPics70Missing} and \ref{CompletionPics80Missing} in this paper).
The interested readers are referred to LRSLibrary \cite{SobralBouwmansZahzah2016-LRSLibrary} (https://github.com/andrewssobral/lrslibrary) for a nice summary of many state-of-the-art algorithms for matrix (tensor) completion and robust PCA.
}

\subsection{Our Contributions}
The main contributions of this paper are as follows. 
First, we introduce several new notions of tensor decomposition for even-order tensors, followed by the new notions of tensor M-rank, symmetric M-rank and strongly symmetric M-rank. Second, we prove the equivalence of these three rank definitions for even-order super-symmetric tensors. Third, we establish the connection between these new ranks and the CP-rank and symmetric CP-rank. Specifically, we show that for a fourth-order tensor, both the CP-rank and the symmetric CP-rank can be lower and upper bounded (up to a constant factor) by the M-rank, and the bound is tight for asymmetric tensors. As a byproduct, we present a class of tensors whose CP-rank can be exactly computed easily. Finally, we solve low-M-rank tensor completion and low-M-rank robust tensor PCA problems for both synthetic and real data. The results demonstrate that our low-M-rank approach outperforms the existing low-n-rank approach, which further confirm that M-rank is a much better approximation to CP-rank.




{\bf Notation.} We use $\C^n$ to denote the $n$-dimensional complex-valued vector space. We adopt calligraphic letter to denote a tensor, i.e.\ $\mathcal{A}=(\mathcal{A}_{i_1i_2\cdots i_d})_{n_1\times n_2\times \cdots\times n_d}$. $\C^{n_1\times n_2\times \cdots\times n_d}$ denotes the space of $d$-th order $n_1\times n_2\times \cdots\times n_d$ dimensional complex-valued tensor. $\pi(i_1,i_2,\cdots,i_d)$ denotes a permutation of indices $(i_1,i_2,\cdots,i_d)$. We use $\ACal_\pi$ to denote the tensor obtained by permuting the indices of $\ACal$ according to permutation $\pi$. Formally speaking, a tensor $\F \in \C^{n_1\times n_2\times \cdots\times n_d}$ is called super-symmetric, if $n_1=n_2=\ldots=n_d$ and $\F = \F_\pi$ for any permutation $\pi$. The space where $\underbrace{n\times n \times \cdots \times n}_{d}$ super-symmetric tensors reside is denoted by $\SI^{n^d}$. We use $\otimes$ to denote the outer product of two tensors; that is, for $\mathcal{A}_1\in \C^{n_1\times n_2\times \cdots\times n_{d}}$ and $\mathcal{A}_2\in \RR^{n_{d+1}\times n_{d+2} \times \cdots\times n_{d+\ell}}$, $\mathcal{A}_1 \otimes \mathcal{A}_2\in\C^{n_1\times n_2\times \cdots\times n_{d+\ell}}$ and
$
(\mathcal{A}_1 \otimes \mathcal{A}_2)_{i_1i_2\cdots i_{d+\ell}} = (\mathcal{A}_1)_{i_1i_2\cdots i_d} (\mathcal{A}_2)_{i_{d+1}\cdots i_{d+\ell}}.
$

\section{M-Rank, Symmetric M-Rank and Strongly Symmetric M-Rank}

In this section, we shall introduce 
the M-decomposition (correspondingly M-rank), the symmetric M-decomposition (correspondingly symmetric M-rank), and the strongly symmetric M-decomposition (correspondingly strongly symmetric M-rank) for tensors, which will be used to provide lower and upper bounds for the CP-rank and the symmetric CP-rank.

\subsection{The M-rank of Even-Order Tensor}
The M-decomposition of an even-order tensor is defined as follows.
\begin{definition}For an even-order tensor $\mathcal{F} \in \C^{n_{1}\times n_{2}\cdots \times n_{2d}}$, the M-decomposition is to find some tensors $\mathcal{A}^i \in \C^{n_{1}\times \cdots \times n_{d}}$, $\mathcal{B}^i \in \C^{n_{{d+1}}\times \cdots \times n_{{2d}}}$ with $i=1,\ldots,r$ such that
\begin{equation} \label{M-decom}
\mathcal{F} = \sum_{i=1}^{r} \mathcal{A}^i\otimes \mathcal{B}^i.
\end{equation}
\end{definition}

The motivation for studying this decomposition stems from 
the following novel matricization technique called square unfolding that has been considered in \cite{Jiang-Ma-Zhang-TPCA-2013,Mu-Huang-Wright-Goldfarb-2013,Yang-Feng-Huang-Suykens-2014}.

\begin{definition}\label{def:matrix-unfolding}
The square unfolding of an even-order tensor $\mathcal{F} \in \C^{n_{1}\times n_{2}\cdots \times n_{2d}}$ (denoted by $\matr(\mathcal{F})\in \C^{(n_1\cdots n_{d})\times(n_{d+1}\cdots n_{2d})}$) is a matrix that is defined as
$$
\matr(\mathcal{F})_{k \ell}:=\mathcal{F}_{i_1\cdots i_{2d}},
$$
where
$$
k = \sum\limits_{j=2}^d (i_j-1)\prod\limits_{q=1}^{j-1}n_{q} +i_1,\, 1 \le i_j \le n_j ,\, 1 \le  j \le d,
$$
$$
\ell =\sum\limits_{j=d+2}^{2d} (i_j-1)\prod\limits_{q=d+1}^{j-1}n_{q} +i_{d+1},\, 1 \le i_j \le n_j ,\, d+1 \le  j \le 2d.
$$
\end{definition}
%

In Definition \ref{def:matrix-unfolding}, the square unfolding merges $d$ indices of $\F$ as the row index of $\matr(\F)$, and merges the other $d$ indices of $\F$ as the column index of $\matr(\F)$. In this pattern of unfolding, we observe that the M-decomposition~\eqref{M-decom} can be rewritten as
\[
\matr(\mathcal{F}) = \sum_{i=1}^{r} \mathbf{a}^i (\mathbf{b}^i)^\top,
\]
where ${a}^i  = \vect(\mathcal{A}^i)$, ${b}^i  = \vect(\mathcal{B}^i)$ for $i=1,\ldots,r$, and $\vect(\cdot)$ is the vectorization operator. Specifically, for a given tensor $\mathcal{F} \in C^{n_1\times n_2\cdots \times n_{d}}$, $\vect(\mathcal{F})_{k}:=\mathcal{F}_{i_1\cdots i_{d}}$ with
\[k = \sum\limits_{j=2}^d (i_j-1)\prod\limits_{q=1}^{j-1}n_{q} +i_1, 1 \le i_j \le n_j ,1 \le  j \le d.\]
Therefore, the M-decomposition of $\F$ is exactly the rank-one decomposition of the matrix $\matr(\mathcal{F})$.
Apparently, unless $\F$ is super-symmetric, there are different ways to unfold the tensor $\F$ by permuting the $2d$ indices. Taking this into account, we now define two types of M-rank (namely, \ M$^+$-rank and M$^-$-rank) of an even-order tensor as follows.

\begin{definition}\label{def:asymmetric-matrix-rank} Given an even-order tensor $\mathcal{F} \in \C^{n_{1}\times n_{2}\cdots \times n_{2d}}$, its
$M^-$-rank (denoted by $\rank_{M^-}(\mathcal {F})$) is the smallest rank of all possible square unfolding matrices, i.e.,
\begin{equation}\label{asymmetric-matrix-decomposition-}
\rank_{M^-}(\mathcal {F}) = \min\limits_{\pi \in \Pi(1,\ldots,2d)}\rank \left(\matr(\mathcal{F}_{\pi})\right),
\end{equation}
where $\Pi(1,\ldots,2d)$ denotes the set of all possible permutations of indices $(1,\ldots,2d)$, and $\mathcal{F}_{\pi}$ is the tensor obtained by permuting the indices of $\F$ according to permutation $\pi$.
In other words, $\rank_{M^-}(\mathcal {F})$ is the smallest
integer $r$ such that
$\mathcal{F_{\pi}} = \sum_{i=1}^{r} \mathcal{A}^i\otimes \mathcal{B}^i,$
holds for some permutation $\pi \in \Pi(1,\ldots,2d)$,
$\mathcal{A}^i \in \C^{n_{i_1}\times \cdots \times n_{i_d}}$, $\mathcal{B}^i \in \C^{n_{i_{d+1}}\times \cdots \times n_{i_{2d}}}$ with $(i_1,\cdots,i_{2d})=\pi(1,\cdots, 2d)$, $i=1,\ldots,r$.
Similarly, the M$^+$-rank (denoted by $\rank_{M^+}(\mathcal {F})$) is defined as the largest rank of all possible square unfolding matrices:
$\rank_{M^+}(\mathcal {F}) = \max\limits_{\pi \in \Pi(1,\ldots,2d)}\rank \left(\matr(\mathcal{F}_{\pi})\right).$
\end{definition}

\subsection{Symmetric M-rank and Strongly Symmetric M-Rank of Even-Order Super-Symmetric Tensor}

Note that if $\mathcal{F}$ is an even-order super-symmetric tensor, $\rank_{M^+}(\mathcal {F}) = \rank_{M^-}(\mathcal {F}) $. In this case, we can simplify the notation without causing any confusion by using $\rank_{M}(\mathcal {F})$ to denote the M-rank of $\mathcal{F}$.

As we mentioned earlier, the decomposition~\eqref{M-decom} is essentially based on the matrix rank-one decomposition of the matricized tensor. In the matrix case, it is clear that there are different ways to decompose a symmetric matrix; for instance,
\[
a b^\top + b a^\top =  \frac{1}{2} (a+b) (a+b)^\top - \frac{1}{2} (a-b) (a-b)^\top.
\]
In other words, a given symmetric matrix may be decomposed as a sum of {\it symmetric}\/ rank-one terms, as well as a sum of {\it non-symmetric}\/ rank-one terms, however yielding the same rank: the minimum number of respective decomposed terms. A natural question arises when dealing with tensors: Does the same property hold for the super-symmetric tensors? The M-decomposition of a tensor is in fact subtler: the decomposed terms can be restricted to symmetric products, and they can also be further restricted to be super-symmetric.

Therefore, we can define the symmetric M-rank and the strongly symmetric M-rank for even-order super-symmetric tensor as follows.
\begin{definition}\label{def:symmetric-matrix-rank}
For an even-order super-symmetric tensor $\mathcal{F} \in \CSI^{n^{2d}}$, its symmetric M-decomposition is defined as
\begin{equation}\label{formula:symmetric-matrix-rank}
\mathcal{F} = \sum_{i=1}^{r}\mathcal{B}^i\otimes \mathcal{B}^i, \quad \mathcal{B}^i \in \C^{n^d},i=1,\ldots,r.
\end{equation}
The symmetric M-rank of $\mathcal {F}$ (denoted by $\rank_{SM}(\mathcal {F})$) is the rank of the symmetric matrix $\matr(\mathcal{F})$,
i.e., $\rank_{SM}(\mathcal {F}) = \rank\left(\matr(\mathcal{F})\right) = \rank_M\left(\mathcal {F}\right)$; or equivalently $\rank_{SM}(\F)$ is the smallest integer $r$ such that \eqref{formula:symmetric-matrix-rank} holds.

In a similar vein, the strongly symmetric M-decomposition is defined as
\begin{equation}\label{formula:strongly-symmetric-matrix-rank}
\mathcal{F} = \sum_{i=1}^{r}\mathcal{A}^i\otimes \mathcal{A}^i,\quad \mathcal{A}^i \in \CSI^{n^d}, i=1,\ldots,r,
\end{equation}
and the strongly symmetric M-rank of $\mathcal {F}$ (denoted by $\rank_{SSM}(\mathcal {F})$) is defined as the smallest integer $r$ such that \eqref{formula:strongly-symmetric-matrix-rank} holds.
\end{definition}

The fact that the M-rank and the symmetric M-rank of an even-order super-symmetric tensor are always equal follows from the similar property of the symmetric matrices. (Note however the M-decompositions may be different). Interestingly,
we can show that $\rank_{SM}(\F)=\rank_{SSM}(\F)$ for any even-order super-symmetric tensor $\F$, which appears to be a new property of the super-symmetric even-order tensors.

\subsection{Equivalence of Symmetric M-Rank and Strongly Symmetric M-Rank}\label{sec:symmetric-rank-equivalent}

To show the equivalence of 
the symmetric M-rank and the strongly symmetric M-rank, we need to introduce the concept of partial symmetric tensors and some lemmas first.

\begin{definition}\label{def:partial-sym}
We say a tensor $\mathcal{F} \in \C^{n^d}$ partial symmetric with respect to indices $\{1,\ldots,m\}$, $m<d$, if
$$
\mathcal{F}_{i_1,\ldots, i_m, i_{m+1},\ldots, i_d} = \mathcal{F}_{\pi(i_1,\ldots, i_m), i_{m+1},\ldots, i_d}, \quad \forall \pi \in \Pi(1,\ldots,m).
$$
We use $\pi_{i,j} \in \Pi(1,\cdots,d)$ to denote the specific permutation that exchanges the $i$-th and the $j$-th indices and keeps other indices unchanged.
\end{definition}

The following result holds directly from Definition \ref{def:partial-sym}.
\begin{lemma}\label{lemma:symmetric+1} Suppose tensor $\mathcal{F} \in \C^{n^d}$ is partial symmetric with respect to indices $\{1,\ldots, m\}$, $m < d$. Then the tensor
\begin{equation*}
\mathcal{F}  + \sum\limits_{j=1}^{m}\mathcal{F}_{\pi_{j,m+1}}
\end{equation*}
is partial symmetric with respect to indices $\{1,\ldots,m+1\}$.
Moreover, it is easy to verify that for $\ell \le k \le m$,
\begin{eqnarray}
  & \left( \sum\limits_{j=1}^{k}\left(\mathcal{F}  - \mathcal{F}_{\pi_{j,m+1}}\right)\right)_{\pi_{\ell,m+1}} \nonumber \\
= & k\cdot\mathcal{F}_{\pi_{\ell,m+1}} - \sum\limits_{j\neq \ell} \mathcal{F}_{\pi_{j,m+1}} - \mathcal{F} \nonumber  \\
= & -k\left(\mathcal{F} -\mathcal{F}_{\pi_{\ell,m+1}} \right) + \sum\limits_{j\neq \ell} \left(\mathcal{F}-\mathcal{F}_{\pi_{j,m+1}}\right). \label{formula:tensor-skewness}
\end{eqnarray}
\end{lemma}

We are now ready to prove the following key lemma.
\begin{lemma}\label{lemma:symmetric-rank-equivalence} Suppose $\mathcal{F} \in \CSI^{n^{2d}}$ and
$\mathcal{F}=\sum_{i=1}^{r} \mathcal{B}^i \otimes \mathcal{B}^i$ with $\mathcal{B}^i \in \C^{n^{d}}$ is partial symmetric with respect to $\{1,\ldots,m\}, m < d$.
Then there exist tensors $\mathcal{A}^i \in  \C^{n^{d}}$, which is partial symmetric with respect to $\{1,\ldots,m+1\}$, for $i=1,\ldots, r$, such that
$$
\mathcal{F}=\sum_{i=1}^{r} \mathcal{A}^i \otimes \mathcal{A}^i.
$$
\end{lemma}

\begin{proof} Define $\mathcal{A}^i = \frac{1}{m+1}\left(\mathcal{B}^i  + \sum\limits_{j=1}^{m}\mathcal{B}^i_{\pi_{j,m+1}}\right)$. From Lemma \ref{lemma:symmetric+1} we know that
$\mathcal{A}^i$ is partial symmetric with respect to $\{1,\ldots,m+1\}$, for $i=1,\ldots, r$.
It is easy to show that
$$
\mathcal{B}^i = \mathcal{A}^i +  \sum\limits_{j=1}^{m}\mathcal{C}^i_j,\;\mbox{with}\quad \mathcal{C}^i_j = \frac{1}{m+1} \left(\mathcal{B}^i - \mathcal{B}^i_{\pi_{j,m+1}}\right).
$$
Because $\mathcal{F}$ is super-symmetric, we have $\mathcal{F} = \mathcal{F}_{\pi_{d+1,d+m+1}} = \mathcal{F}_{\pi_{1,m+1}} = (\mathcal{F}_{\pi_{1,m+1}})_{\pi_{d+1,d+m+1}} $, which implies
\begin{align}
\mathcal{F} & = \sum_{i=1}^{r} \mathcal{B}^i \otimes \mathcal{B}^i = \sum_{i=1}^{r} \mathcal{B}^i \otimes \mathcal{B}^i_{\pi_{1,m+1}} \nonumber \\ & =\sum_{i=1}^{r} \mathcal{B}^i_{\pi_{1,m+1}} \otimes \mathcal{B}^i  =\sum_{i=1}^{r} \mathcal{B}^i_{\pi_{1,m+1}} \otimes \mathcal{B}^i_{\pi_{1,m+1}}.\label{formula:decomposition-symmetric}
\end{align}
By using \eqref{formula:tensor-skewness}, we have
\begin{equation}\label{Lemma3.3-9}
\mathcal{B}^i_{\pi_{1,m+1}} = \left(\mathcal{A}^i +  \sum\limits_{j=1}^{m}\mathcal{C}^i_j \right)_{\pi_{1,m+1}} = \mathcal{A}^i +  \sum\limits_{j=2}^{m}\mathcal{C}^i_j - m \cdot\mathcal{C}^i_1.
\end{equation}
Combining \eqref{Lemma3.3-9} and \eqref{formula:decomposition-symmetric} yields
\begin{align}
 & \mathcal{F} \nonumber \\
=&  \sum_{i=1}^{r} \mathcal{B}^i \otimes \mathcal{B}^i = \sum_{i=1}^{r} (\mathcal{A}^i +  \sum\limits_{j=1}^{m}\mathcal{C}^i_j )\otimes (\mathcal{A}^i +  \sum\limits_{j=1}^{m}\mathcal{C}^i_j ) \label{equation-skewness1} \\
=& \sum_{i=1}^{r} \mathcal{B}^i \otimes \mathcal{B}^i_{\pi_{1,m+1}} \nonumber \\
=&  \sum_{i=1}^{r} (\mathcal{A}^i +  \sum\limits_{j=1}^{m}\mathcal{C}^i_j )\otimes ( \mathcal{A}^i +  \sum\limits_{j=2}^{m}\mathcal{C}^i_j - m \cdot\mathcal{C}^i_1 )\label{equation-skewness2} \\
=& \sum_{i=1}^{r} \mathcal{B}^i_{\pi_{1,m+1}} \otimes \mathcal{B}^i \nonumber \\
=& \sum_{i=1}^{r} ( \mathcal{A}^i +  \sum\limits_{j=2}^{m}\mathcal{C}^i_j - m \cdot\mathcal{C}^i_1 ) \otimes (\mathcal{A}^i +  \sum\limits_{j=1}^{m}\mathcal{C}^i_j )\label{equation-skewness3} \\
=& \sum_{i=1}^{r} \mathcal{B}^i_{\pi_{1,m+1}} \otimes \mathcal{B}^i_{\pi_{1,m+1}} \nonumber \\
=& \sum_{i=1}^{r} ( \mathcal{A}^i +  \sum\limits_{j=2}^{m}\mathcal{C}^i_j - m \cdot\mathcal{C}^i_1 ) \otimes  ( \mathcal{A}^i +  \sum\limits_{j=2}^{m}\mathcal{C}^i_j - m \cdot\mathcal{C}^i_1 ).\label{equation-skewness4}
\end{align}
It is easy to check that
$
\frac{\eqref{equation-skewness4}+ m\times \eqref{equation-skewness3}+m\times\eqref{equation-skewness2}+m^2\times\eqref{equation-skewness1} }{(1+m)^2}
$
yields
$\mathcal{F} = \sum_{i=1}^{r} \left(\mathcal{A}^i +  \sum\limits_{j=2}^{m}\mathcal{C}^i_j \right)\otimes \left(\mathcal{A}^i +  \sum\limits_{j=2}^{m}\mathcal{C}^i_j \right).
$
Then we repeat this procedure. That is, since $\F\in\CSI^{n^{2d}}$, we have $\mathcal{F} = \mathcal{F}_{\pi_{d+2,d+m+1}} = \mathcal{F}_{\pi_{2,m+1}} = (\mathcal{F}_{\pi_{2,m+1}})_{\pi_{d+2,d+m+1}}$. By letting $\mathcal{B}^i=\mathcal{A}^i +  \sum\limits_{j=2}^{d}\mathcal{C}^i_j$, we can apply the same procedure as above to obtain $\mathcal{F} = \sum_{i=1}^{r} \left(\mathcal{A}^i +  \sum\limits_{j=3}^{m}\mathcal{C}^i_j \right)\otimes \left(\mathcal{A}^i +  \sum\limits_{j=3}^{m}\mathcal{C}^i_j \right)$. We just repeat this procedure until
$\mathcal{F} = \sum_{i=1}^r \mathcal{A}^i \otimes \mathcal{A}^i$ and this completes the proof.
\end{proof}


Now we are ready to present the equivalence of symmetric M-rank and strongly symmetric M-rank.

\begin{theorem}\label{theorem:symmetric-rank-equivalence}
For an even-order super-symmetric tensor $\mathcal{F} \in \CSI^{n^{2d}}$, its M-rank, symmetric M-rank and strongly symmetric M-rank are the same, i.e. $\rank_{M}(\mathcal{F})=\rank_{SM}(\mathcal{F})=\rank_{SSM}(\mathcal{F})$.
\end{theorem}

\begin{proof} The equality $\rank_{M}(\mathcal{F})=\rank_{SM}(\mathcal{F})$ follows directly from the definition of symmetric M-rank. We thus only need to prove $\rank_{SM}(\mathcal{F})=\rank_{SSM}(\mathcal{F})$. Suppose $\rank_{SM}(\mathcal{F} ) = r$, which means there exist $\mathcal{B}^i \in \C^{n^d}, i=1,\ldots,r$, such that $\mathcal{F}  = \sum\limits_{i=1}^{r}\mathcal{B}^i \otimes \mathcal{B}^i$. By applying Lemma \ref{lemma:symmetric-rank-equivalence} at most $d$ times, we can find super-symmetric tensors $\mathcal{A}^i \in \CSI^{n^d}, i=1,\ldots,r$ such that
$\mathcal{F}  = \sum\limits_{i=1}^{r}\mathcal{A}^i \otimes \mathcal{A}^i$. Hence, we have $\rank_{SSM}(\mathcal{F}) \le r = \rank_{SM}(\mathcal{F})$. On the other hand, it is obvious that $\rank_{SM}(\mathcal{F}) \le \rank_{SSM}(\mathcal{F})$. Combining these two inequalities yields $\rank_{SM}(\mathcal{F})=\rank_{SSM}(\mathcal{F})$.
\end{proof}


\section{Bounding CP-Rank for Even-Order Tensor using M-rank}\label{sec:approximate-CP-rank}

In this section, we analyze the relation between the CP-rank and the M-rank. Specifically, for even-order tensor, we establish the equivalence between the symmetric CP-rank and the M-rank under the rank-one assumption. Then we particularly focus on the
fourth-order tensors, because many multi-dimensional data from real practice are in fact fourth-order tensors. For example, the colored video completion and decomposition problems considered in \cite{Gandy-Recht-Yamada-2011,Goldfarb-Qin-tensor-2013,LMWY09} can be formulated as low-CP-rank fourth-order tensor recovery problems. We show that the CP-rank and the symmetric CP-rank for fourth-order tensor can be both lower and upper bounded (up to a constant factor) by the corresponding M-rank.

\subsection{Rank-One Equivalence for Super-Symmetric Even-Order Tensor}

In our previous work \cite{Jiang-Ma-Zhang-TPCA-2013}, we already showed that if a super-symmetric even-order tensor $\mathcal{F}$ is real-valued and the decomposition is performed in the real domain, then $\rank_{CP}(\F) =1 \Longleftrightarrow \rank(\matr(\F))=1$. Here we show that a similar result can be established when $\mathcal {F}$ is complex-valued and the decomposition is performed in the complex domain. To see this, we first present the following lemma.
\begin{lemma}\label{lemma:rank-one}If a $d$-th order tensor $\mathcal{A} = a^1\otimes a^2 \otimes  \cdots \otimes a^d \in \CSI^{n^{d}}$ is super-symmetric, then we have {$a^i =c_i a^1$ for some $c_i \in \C, \forall \; i=2,\ldots,d$} and
$\mathcal{A} = \underbrace{b\otimes b \otimes  \cdots \otimes b}_{d}$ for some $b \in \C^{n}$.
\end{lemma}
\begin{proof} Since $\mathcal{A}$ is super-symmetric, construct $\mathcal{ T} = \bar{\mathcal{A}} \otimes \mathcal{A}$ and it is easy to show that $\forall\,(j_1\dots j_d) \in \Pi (i_1\dots i_d),\,(j_{d+1}\dots j_{2d})\in\Pi(i_{d+1}\dots i_{2d}),$
$$\mathcal{ T}_{i_1\dots i_d i_{d+1}\dots i_{2d}} = \mathcal{ T}_{j_1\dots j_d j_{d+1}\dots j_{2d}},$$
and $\forall 1\le i_1\le\dots\le i_{d}\le n,\, 1\le i_{d+1}\le\dots\le i_{2d}\le n,$
$$\mathcal{ T}_{i_1\dots i_d i_{d+1}\dots i_{2d}} = \overline{\mathcal{ T}_{i_{d+1}\dots i_{2d} i_1\dots i_d}}.$$
Therefore, $\mathcal T$ belongs to the so-called conjugate partial symmetric tensor introduced in~\cite{Jiang2014a}. Moreover, from Theorem 6.5 in~\cite{Jiang2014a}, we know that
\begin{align*}
& \max_{\|x\|=1 } \mathcal{ T}(\underbrace{\overline{x},\dots,\ov{x}}_d,\underbrace{x,\dots,x}_d) \\
= & \max_{\|x^i\|=1,\,i=1,\ldots,d} \mathcal{ T}(\overline{x^1},\dots,\ov{x^d},x^{1},\dots,x^{d}) \\
= & \|a^1\|^2\cdot \|a^2\|^2\cdot \cdots \cdot \|a^d\|^2.
\end{align*}
So there must exist an $\hat{x}$ with $\|\hat{x}\|=1$ such that $ |(a^i)^{\top}\hat{x}|=\|a^i\|$ for all $i$, {which implies that $a^i =c_i a^1$ for some $c_i \in \C, \forall \; i=2,\ldots,d$, and $\mathcal{A} =\lambda\, \underbrace{a^1\otimes a^1 \otimes  \cdots \otimes a^1}_{d}$ for some constant $\lambda = \Pi_{i=2}^{d} c_i$. Finally by taking $b = \sqrt[d]{\lambda}a^1$, the conclusion follows.}
\end{proof}

The rank-one equivalence is established in the following theorem.
\begin{theorem}\label{theorem:tensor-rank1}
Suppose $\mathcal{F} \in \CSI^{n^{2d}}$ and  we have
$$\rank_{M}(\mathcal {F}) = 1 \Longleftrightarrow \rank_{SCP}(\mathcal {F})=1.$$
\end{theorem}
\begin{proof}
Suppose $\rank_{SCP}(\mathcal{F})=1$ and $\mathcal{F}=\underbrace{x\otimes \cdots \otimes x}_{2d}$ for some $x \in \C^n$. By constructing $\mathcal{A}=\underbrace{x\otimes \cdots \otimes x}_{d}$, we have $\mathcal{F}=\mathcal{A} \otimes \mathcal{A}$ with $\mathcal{A} \in \CSI^{n^{d}}$. Thus, $\rank_{M}(\mathcal {F})=1$.

To prove the other direction, suppose that we have $\mathcal{F} \in \SI^{n^{2d}}$ and its M-rank is one, i.e.\ $\mathcal{F}=\mathcal{A} \otimes \mathcal{A}$ for some $\mathcal{A} \in \CSI^{n^{d}}$. By similar arguments as in Lemma 2.1 and Proposition 2.3 in~\cite{Jiang-Ma-Zhang-TPCA-2013}, one has that the Tucker rank of $\mathcal{A}$ is $(1,1,\cdots,1)$ and consequently the asymmetric CP-rank of $\mathcal{A}$ is one. This fact together with Lemma~\ref{lemma:rank-one} implies that the symmetric CP-rank of $\mathcal{A}$ is one as well, i.e., $\mathcal{A} = \underbrace{b \otimes  \cdots \otimes b}_{d}$ for some $b \in \C^{n}$. It follows from $\mathcal{F}=\mathcal{A} \otimes \mathcal{A} = \underbrace{b\otimes \cdots \otimes b}_{2d}$ that $\rank_{SCP}(\mathcal {F})=1$.
\end{proof}

\subsection{Bounds for Asymmetric Fourth-Order Tensors}

For an asymmetric fourth-order tensor, the relation between its CP-rank and the corresponding M-rank is summarized in the following result.

\begin{theorem}\label{theorem:asymmetric-bound} Suppose $\mathcal {F} \in  \C^{n_{1}\times n_{2}\times n_{3}\times n_{4}}$ with $n_1 \le n_2 \le n_3 \le n_4$. Then for any permutation $\pi$ of $(1, 2, 3, 4)$ it holds that
\begin{equation}\label{asymmetric-bound1}
\rank( \matr(\mathcal {F}_{\pi})) \le \rank_{CP} (\mathcal {F}_{\pi}) \le n_1n_3\cdot \rank( \matr(\mathcal {F}_{\pi})).
\end{equation}
Moreover, the inequalities above can be sharpened to
\begin{equation}\label{asymmetric-bound2}
\rank_{M^+} (\mathcal {F}) \le \rank_{CP} (\mathcal {F}) \le n_1n_3\cdot\rank_{M^-} (\mathcal {F}).
\end{equation}
\end{theorem}
\begin{proof} Suppose $\rank_{CP} (\mathcal {F}) = r$. Let the rank-one decomposition be
$\mathcal{F} = \sum_{i=1}^{r} a^{1,i}\otimes a^{2,i}\otimes a^{3,i} \otimes a^{4,i}$, with $a^{k,i} \in \C^{n_i}$ for $k=1,\ldots, 4$ and $i=1,\dots, r.$
By letting $A^i = a^{1,i}\otimes a^{2,i}$ and $B^i = a^{3,i} \otimes a^{4,i}$, we get $\mathcal{F} = \sum_{i=1}^{r} A^i \otimes B^i$. Thus $\rank_M (\mathcal {F}) \le r = \rank_{CP} (\mathcal {F}) $. 
In fact, this holds for $\mathcal {F}_{\pi}$ where $\pi$ is any permutation of $(1, 2, 3, 4)$.

On the other hand, for any $\mathcal {F}_{\pi}$  denote $r_M = \rank(\matr(\mathcal {F}_{\pi}))$ and $(j_1,j_2,j_3,j_4)=\pi(1,2,3,4)$. Then
$\mathcal{F}_{\pi} = \sum_{i=1}^{r_M} A^i \otimes B^i$ with matrices $A^i \in \C^{n_{j_1} \times n_{j_2}}$, $B^i \in \C^{n_{j_3} \times n_{j_4}}$ for $i=1,\ldots,r_M,$
and it follows that $\rank (A^i) \le \ell_1 $ and $\rank (B^i) \le \ell_2$ for all $i=1,\ldots, r_M$, where $\ell_1 := \min\{n_{j_1},n_{j_2}\}$ and $\ell_2 := \min\{n_{j_3},n_{j_4}\}$. In other words, matrices $A^i$ and $B^i$ admit some rank-one decompositions with at most $\ell_1$ and $\ell_2$ terms, respectively. Consequently, $\mathcal{F}$ can be decomposed as the sum of at most $r_M\ell_1\ell_2$ rank-one tensors, or equivalently
\begin{align*}
\rank_{CP} (\mathcal {F}_{\pi}) & \le  \min\{n_{j_1},n_{j_2}\}\cdot \min\{n_{j_3},n_{j_4}\}\cdot \rank_M (\mathcal {F}_{\pi}) \\ & \le n_1n_3\cdot\rank_M (\mathcal {F}_{\pi}).
\end{align*}
Since the bounds~\eqref{asymmetric-bound1} hold for all $\mathcal {F}_{\pi}$ and
$\rank_{M^-} (\mathcal {F})= \min_\pi{\rank( \matr(\mathcal {F}_{\pi}))}$, $\rank_{M^+} (\mathcal {F})= \max_\pi{\rank( \matr(\mathcal {F}_{\pi}))},$
the sharper bounds \eqref{asymmetric-bound2} follow immediately.
\end{proof}
The following results further show that the bounds in \eqref{asymmetric-bound2} are actually tight.
\begin{proposition}\label{tight-example}
Let us consider a fourth order tensor $\mathcal{F} \in \C^{n_1 \times n_2 \times n_3 \times n_4} $ such that
\begin{equation}\label{special-tensor}
\mathcal{F} = A \otimes B \;\mbox{for some matrices}\;A \in \C^{n_1 \times n_2}\; \mbox{and} \;B \in \C^{n_3 \times n_4}.
\end{equation}
Denote $r_1= \rank(A)$, $r_2= \rank(B)$. Then, the following holds: 
\begin{itemize}
\item[(i)] The Tucker rank of $\mathcal{F}$ is $(r_1, r_1, r_2, r_2)$;
\item[(ii)] 
$\rank_{M^+}(\mathcal{F}) = r_1\, r_2$ and $\rank_{M^-}(\mathcal{F}) = 1$;
\item[(iii)] 
$\rank_{CP} (\mathcal{F}) = r_1\, r_2$.
\end{itemize}
\end{proposition}
\begin{proof}
Suppose the singular value decompositions of $A$ and $B$ are given by
\begin{equation}\label{svd-matrix}
A = \sum_{i=1}^{r_1} a^i \otimes b^i\; \mbox{and}\; B = \sum_{j=1}^{r_2} c^j \otimes d^j.
\end{equation}
Recall that $F(1)$ denotes the mode-$1$ unfolding of $\mathcal{F}$. According to~\eqref{svd-matrix}, it is easy to verify that
\begin{equation}\label{mode-1-ortho-decomp}F(1) = \sum_{i=1}^{r_1} a^i \otimes \vect(b^i \otimes B).\end{equation}
Moreover we observe that for $i \neq j$, $(\vect(b^i \otimes B))^{\top}(\vect(b^j \otimes B)) = (b^i)^{\top}b^j\cdot \tr(B^{\top}B) = 0$. Thus, \eqref{mode-1-ortho-decomp} is indeed an orthogonal
decomposition of $F(1)$ and thus, $\rank(F(1)) = r_1$.
Similarly we can show that $\rank(F(2)) = r_1$, $\rank(F(3)) = r_2$ and $\rank(F(4)) = r_2$. This proves part (i).

Now we consider the square unfoldings of $\mathcal{F}$. Let $F(1,2)$, $F(1,3)$, $F(1,4)$ be the square unfolded matrices by grouping indices $(1,2)$, $(1,3)$, $(1,4)$ respectively.
Due to \eqref{special-tensor}, we immediately have that $\rank(F(1,2))=1$ and also
\begin{equation}\label{square-unfolding-ortho-decomp}
F(1,3) = \sum_{i=1}^{r_1}\sum_{j=1}^{r_2}\vect(a^i \otimes c^j)\otimes \vect(b^i \otimes d^j).
\end{equation}
From the orthogonality of $a^i$'s, $b^i$'s, $c^j$'s, and $d^j$'s, it follows that $\{\vect(a^i \otimes c^j)\}_{i,j}$ and $\{\vect(b^i \otimes d^j)\}_{i,j}$ are two orthogonal bases. In other words, \eqref{square-unfolding-ortho-decomp} is an orthogonal decomposition of $F(1,3)$ and thus $\rank(F(1,3)) = r_1\,r_2$. In the same vein we can show that $\rank(F(1,4)) = r_1\,r_2$ as well. Since $F(1,2)$, $F(1,3)$ and $F(1,4)$ form all the square unfoldings of $\mathcal{F}$, we can conclude that $\rank_{M^+}(\F) = \rank(F(1,3))  = \rank(F(1,4)) =  r_1\,r_2$ and $\rank_{M^-}(\F) = 1$. This proves part (ii).

Finally, since \eqref{special-tensor} also implies
$$
\mathcal{F} = \sum_{i=1}^{r_1}\sum_{j=1}^{r_2}a^i \otimes b^i \otimes c^j \otimes d^j,
$$
it then follows that $\rank_{CP}(\mathcal{F}) \le r_1\,r_2$. On the other hand, note that the first inequality in~\eqref{asymmetric-bound1} holds for any square unfoldings.
Combining this with $\rank(F(1,3)) = r_1\,r_2$, one concludes that $r_1\,r_2$ is a lower bound of $ \rank_{CP}\mathcal{F}$ as well, hence
$\rank_{CP}(\mathcal{F}) = r_1\,r_2$.
\end{proof}
\begin{remark}
Now we remark that the bounds in~\eqref{asymmetric-bound2} are tight. Suppose tensor $\mathcal{F}$ is given in the form of \eqref{special-tensor} with $n_1 \le n_2 \le n_3 \le n_4$, and $\rank(A)=n_1$, $\rank(B)=n_3$. According to the above results, we have $\rank_{M^-}(\mathcal{F}) = 1$ and $\rank_{M^+}(\mathcal{F}) = \rank_{CP}(\mathcal{F}) = n_1 n_3$, which imply that the both lower bound and upper bound in~\eqref{asymmetric-bound2} are essentially tight. Moreover, {in this example the Tucker rank is exactly $(n_1,n_1,n_3,n_3)$, which in turn shows that $\rank_{M^+}$ is a superior approximation of $\rank_{CP}$ in this particular case. Although we are unable to extend such claim to the more general setting, we present a few examples in the numerical part for which $\rank_{CP}$ is strictly larger than any component of Tucker rank but is essentially identical to $\rank_{M^+}$.} In addition, it is easy to show that by similar argument the bounds in \eqref{asymmetric-bound1} also hold for the ranks defined for real-valued decompositions, i.e., $\rank^{\Br}( \matr(\mathcal {F}_{\pi})) \le \rank^{\Br}_{CP} (\mathcal {F}_{\pi}) \le n_1n_3\cdot \rank^{\Br}( \matr(\mathcal {F}_{\pi}))$. Moreover, for matrix $\matr(\mathcal {F}_{\pi})$ we have $\rank( \matr(\mathcal {F}_{\pi}))  = \rank^{\Br}( \matr(\mathcal {F}_{\pi})) $, thus establishing the following bounds: 
\begin{align*}
\rank_{CP} (\mathcal {F}_{\pi}) & \le \rank^{\Br}_{CP} (\mathcal {F}_{\pi}) \le n_1n_3\cdot \rank^{\Br}( \matr(\mathcal {F}_{\pi})) \\ & = n_1n_3\cdot\rank( \matr(\mathcal {F}_{\pi}))\le n_1n_3\cdot\rank_{CP} (\mathcal {F}_{\pi}).
\end{align*}
\end{remark}
Proposition~\ref{tight-example} can be further extended to exactly compute the CP-rank for a class of tensors.
\begin{corollary}
Consider an even order tensor $\mathcal{F} \in \C^{n_1 \times n_2 \times \cdots \times n_{2d}} $ such that
\begin{equation*}
\mathcal{F} = A^1 \otimes A^2 \otimes \cdots \otimes A^d \;\mbox{for some matrices}\;A^i \in \C^{n_{2i-1} \times n_{2i}}.
\end{equation*}
Denoting $r_i= \rank(A^i)$ for $i=1,\ldots, d$, we have that $\rank_{CP}( \mathcal{F}  ) = \rank_{M^+}( \mathcal{F}  ) = r_1r_2 \cdots r_d$.
\end{corollary}

\subsection{Bounds for Super-Symmetric Fourth-Order Tensors}

Theorem~\ref{theorem:tensor-rank1} essentially states that the M-rank and the symmetric CP-rank are the same in the rank-one case.  This equivalence, however, does not hold in general. 
In this subsection, we show that although they are not equivalent, the symmetric CP-rank of $\F \in\CSI^{n^{2d}}$ can be both lower and upper bounded (up to a constant factor) by the corresponding M-rank.

\begin{theorem}\label{rank-upper-bound} For any given $\mathcal{F}\in \CSI^{n^4}$, it holds that
$$
\rank_{M}(\mathcal {F}) \le \rank_{SCP}(\mathcal {F}) \le n^2\, \rank_{M}(\mathcal {F}).
$$
\end{theorem}
\begin{proof}
Let us first prove $\rank_{M}(\mathcal {F}) \le \rank_{SCP}(\mathcal {F})$.
Suppose $\rank_{SCP}(\mathcal {F}) = r$, i.e.,
$$\mathcal{F} = \sum_{i=1}^{r} a^{i}\otimes a^{i}\otimes a^{i} \otimes a^{i}\; \mbox{with} \;a^{i} \in \C^{n}\; \mbox{for}\; i=1,\dots, r.$$
By letting $A^i = a^{i}\otimes a^{i}$, we get $\mathcal{F} = \sum_{i=1}^{r} A^i \otimes A^i$ with $A^i \in \SI^{n^2}$. Thus $\rank_{M} (\mathcal {F}) \le r = \rank_{SCP}(\mathcal {F}) $.

We now prove $\rank_{SCP}(\mathcal {F}) \le (n+4n^2)\, \rank_{M}(\mathcal {F})$. Suppose that $\rank_{M}(\mathcal {F})=r$, then from \eqref{theorem:symmetric-rank-equivalence} it holds that $\mathcal{F} = \sum\limits_{i=1}^{r} A^i \otimes A^i$ with $A^i\in\CSI^{n^2}$. Now consider the associated polynomial $\mathcal{F}(x,x,x,x) = \sum\limits_{i=1}^{r} \left( x^{\top} A^i x \right)^2$.
Since $A^i$ is a complex symmetric matrix,
by letting $y^i = {A^i}^{1/2}x$, we have
\begin{align*}
& \mathcal{F}(x,x,x,x) \\ = & \sum\limits_{i=1}^{r} \left( {y^i}^{\top} y^i \right)^2 =\sum\limits_{i=1}^{r}\sum\limits_{j,k=1}^{n}  (y^i_j)^2 (y^i_k)^2 \\
= & \sum\limits_{i=1}^{r} \left( \sum\limits_{j\le k}^{n} \frac{1}{4}\left( (y^i_j+y^i_k)^4 + (y^i_j-y^i_k)^4  \right) -(\frac{n}{4}-1)\sum\limits_{j=1}^{n}(y^i_j)^4 \right).
\end{align*}
Note that $y^i_j$ is a linear function of $x$ for any $i,j$. Therefore, $(y^i_j+y^i_k)^4$ and $(y^i_j)^4$ correspond to some symmetric rank-one tensors. Therefore, $\rank_{SCP}(\mathcal {F}) \le \left(\frac{2n(n-1) }{2}+n\right)\, r = n^2\, \rank_{M}(\mathcal {F})$.
\end{proof}
We remark that the rank-one decomposition of $(y^i_j)^2 (y^i_k)^2$ is related to the Waring rank of polynomial. By involking Proposition 3.1 in \cite{CarliniCatalisanoGeramita2012} yeilds that $\rank_{SCP}(y^i_j)^2 (y^i_k)^2) \le 3$ and $\rank_{SCP}(\mathcal {F}) \le  \left( \frac{3n(n-1)}{2}+n \right)\, \rank_{M}(\mathcal {F}) =\left( \frac{3n^2-n}{2} \right)\, \rank_{M}(\mathcal {F})$, which is sightly worse than the bound in Theorem \ref{rank-upper-bound}. If the coefficients of the rank-one terms in the decomposition of $\mathcal {F}$ are required to be nonngetive, we can resort to the algorithms in \cite{Alon1986,Jiang2014} to perform such rank-one decomposition.

\section{The Low-M-Rank Tensor Recovery and Robust Tensor PCA}

In this section, we consider the low-rank tensor recovery problem \eqref{prob:tensor-completion} and robust tensor PCA \eqref{prob:robust-tensor-recovery} with an emphasis on the fourth-order tensor case.
{In the context of tensor recovery, M$^+$-rank takes a {\em conservative}\/ attitude towards estimating the CP-rank, because when M$^+$-rank is minimized, the ranks of all other unfolding matrices are also small. As a result, minimizing the M$^+$-rank is like considering the worst case situation. On the other hand, M$^-$-rank is a more {\em optimistic}\/ estimation, because M$^-$-rank is the smallest rank among all unfolding matrices, and the matrix is more likely to be recovered using a matrix completion model.}
On the middle ground, one may choose to work with a pre-specified $\pi$.
In our numerical experiments, for simplicity, for a fourth-order tensor $\F$, we always choose to group the first two indices as the row index, and group the last two indices as the column index for square unfolding (we use $\matr(\F)$ to denote the corresponding matrix).
According to Theorems \ref{theorem:asymmetric-bound} and \ref{rank-upper-bound},  by multiplying a constant factor, $\rank(\matr(\F))$ can provide an upper bound for the CP-rank and the symmetric CP-rank of $\F$ (if it is also super-symmetric). We denote $X = \matr(\mathcal{X})$, $Y = \matr(\mathcal{Y})$ and $F = \matr(\mathcal{F})$. Without loss of generality, we replace the CP-rank in the objective of \eqref{prob:tensor-completion} and \eqref{prob:robust-tensor-recovery} by $\rank(X)$, and it follows from Theorem \ref{theorem:asymmetric-bound} that by minimizing $\rank(X)$, $\rank_{CP}(\XCal)$ will be small as well.
In other words, rather than solving \eqref{prob:tensor-completion} and \eqref{prob:robust-tensor-recovery}, we solve the following two matrix problems
\begin{equation}\label{matrix-rank-min} \min\limits_{X} \ \rank(X), \quad \st, \quad \bar{\linear}(X) = b, \end{equation}
and
\begin{equation}\label{matrix-rpca} \min\limits_{Y,Z} \ \rank(Y) + \lambda \|Z\|_0, \quad \st, \quad Y + Z = F, \end{equation}
where $\bar{\linear}$ is a linear mapping such that $\bar{\linear}(X)=\linear(\XCal)$.

It is now natural to consider the convex relaxations of the two matrix problems \eqref{matrix-rank-min} and \eqref{matrix-rpca}; i.e., we replace the rank function by the nuclear norm and replace the cardinality function by the $\ell_1$ norm. This results in the following two convex relaxations for \eqref{matrix-rank-min} and \eqref{matrix-rpca}:
\begin{equation}\label{matrix-rank-min-nuclear} \min\limits_{X} \ \|X\|_*, \quad \st, \quad \bar{\linear}(X) = b, \end{equation}
and
\begin{equation}\label{matrix-rpca-nuclear} \min\limits_{Y,Z} \ \|Y\|_* + \lambda \|Z\|_1, \quad \st, \quad Y + Z = F. \end{equation}
Note that all the variables in \eqref{matrix-rank-min-nuclear} and \eqref{matrix-rpca-nuclear} are complex-valued. Thus, the $\ell_1$ norm is defined as $\|Z\|_1:=\sum_{ij}\sqrt{(\re(Z_{ij}))^2+(\im(Z_{ij}))^2}$. Although \eqref{matrix-rank-min-nuclear} and \eqref{matrix-rpca-nuclear} are complex-valued problems, they can still be solved by the methods in \cite{Cai-Candes-Shen-2008,Ma-Goldfarb-Chen-2008,Liu-Sun-Toh-2009,Aybat-Goldfarb-Ma-RPCA-2012,Tao-Yuan-SPCP-2011} with minor modifications. We omit the details for succinctness.

When the tensors are super-symmetric, we can impose the super-symmetric constraint and get the following formulation:
$$
\begin{array}{ll} \min\limits_{X} & \|{X}\|_* \\ \st & \bar{\linear}(Y) =  b, \;\matr^{-1}(X) \in \SI^{n^4},
\end{array}
$$
where $\matr^{-1}(X)$ is the tensor whose square unfolding is $X$.
Note that the above problem is equivalent to
\begin{equation}\label{matrix-rank-min-nuclear-sym}
\begin{array}{ll} \min\limits_{X} & \|{X}\|_*  \\ \st & \matr^{-1}(Y) \in \SI^{n^4},\\
                                                                        & \bar{\linear}(Y) =  b, \quad X=Y,
\end{array}
\end{equation}
which can be efficiently solved by the standard alternating direction method of multipliers (see the survey paper~\cite{Boyd2011} for more details).

\section{Numerical Results}

In this section, we shall present the performance of our new low-M-rank tensor completion and robust tensor PCA models.


\subsection{Approximating the CP-rank via the M-rank}
In this subsection, we consider the problem of estimating the CP-rank for some structured tensors. We first construct tensors in the form of
\begin{equation}\label{tensor-generation}
\mathcal{T} = \sum\limits_{i=1}^{r}a^i \otimes b^i \otimes c^i \otimes d^i \in \mathbb{C}^{n_1\times n_2\times n_3\times n_4},
\end{equation}
where $a^i,  b^i,  c^i, d^i, i=1,\ldots ,r$ are randomly generated. Obviously, the CP-rank of the resulting tensor $\mathcal{T}$ is less than or equal to $r$.
For each set of $(n_1,n_2,n_3,n_4)$ and $r$, we randomly generate $20$ tensors according to~\eqref{tensor-generation}. Then we compute the Tucker rank, the M$^+$-rank and the M$^-$-rank of each tensor, and report their average value over the 20 instances in Table~\ref{CP-approximation-tab1}.

\begin{table}[ht]{\footnotesize
\caption{CP-rank approximation: M-rank vs Tucker rank for tensors in \eqref{tensor-generation}}
\begin{center}
\begin{tabular}
{|c|c|c|c|}\hline
$r$ & Tucker rank & $M^+$ & $M^-$ \\ \hline
\multicolumn{4}{|c|}{Dimension $10 \times 10 \times 10 \times 10$}\\ \hline
12 &(10,10,10,10) & 12 &12 \\ \hline
  \multicolumn{4}{|c|}{Dimension $10 \times 10 \times 15 \times 15$}\\ \hline
12 &(10,10,12,12) & 12 &12 \\ \hline
    \multicolumn{4}{|c|}{Dimension $15 \times 15 \times 15 \times 15$}\\ \hline
18 &(15,15,15,15) & 18 &18 \\ \hline
    \multicolumn{4}{|c|}{Dimension $15 \times 15 \times 20 \times 20$}\\ \hline
18 &(15,15,18,18) & 18 &18 \\ \hline
    \multicolumn{4}{|c|}{Dimension $20 \times 20 \times 20 \times 20$}\\ \hline
30 &(20,20,20,20) & 30 & 30 \\ \hline
    \multicolumn{4}{|c|}{Dimension $20 \times 20 \times 25 \times 25$}\\ \hline
30 &(20,20,25,25) & 30 & 30 \\ \hline
    \multicolumn{4}{|c|}{Dimension $25 \times 25 \times 30 \times 30$}\\ \hline
40 &(25,25,30,30) & 40 & 40 \\ \hline
    \multicolumn{4}{|c|}{Dimension $25 \times 25 \times 30 \times 30$}\\ \hline
40 &(30,30,30,30) & 40 & 40 \\ \hline
  \end{tabular}
\end{center}
\label{CP-approximation-tab1}}
\end{table}

From Table \ref{CP-approximation-tab1}, we see that for all instances $\rank_{M^+}(\mathcal{T}) = \rank_{M^-}(\mathcal{T})=r$. Thus by Theorem~\ref{theorem:asymmetric-bound}, $\rank_{M^+}(\mathcal{T}) \le \rank_{CP}(\mathcal{T}) $, and from the construction of $\mathcal{T}$, $\rank_{CP}(\mathcal{T}) \le r$. Therefore,
we conclude that the CP-rank of these tensors is exactly $r$ and the M-rank equals 
the CP-rank for these random instances. Moreover,
since $r$ is chosen to be larger than one dimension of the tensor, some components of the Tucker rank are strictly less than $r$.

In another setting, we generate tensors in the following manner
\begin{equation}\label{tensor-generation2}
\mathcal{T} = \sum\limits_{i=1}^{r}A^i \otimes B^i,
\end{equation}
where matrices $A^i,B^i, i=1,\ldots,r$ are randomly generated in such a way that $\rank (A^i) = \rank(B^i) = k, i=1,\ldots,r$. Consequently, $r k^2$ is an upper bound for the CP-rank of $\mathcal{T}$.
From Proposition~\ref{tight-example}, we know that $\rank_{M^+}(\mathcal{T}) = \rank_{CP}(\mathcal{T}) = k^2$ when $r=1$.
One may wonder if $\rank_{M^+}(\mathcal{T}) = \rank_{CP}(\mathcal{T})$ when $r > 1$.
To this end, we let $r=k =2, 3, 4, 5$ and
generate $20$ random instances for different choices of $r$, $k$ and tensor dimensions. For each instance, we compute its Tucker rank,  the M$^+$-rank and the M$^-$-rank, and report the average value over these 20 instances in Table~\ref{CP-approximation-tab2}.

\begin{table}[ht]{\footnotesize
\caption{CP-rank approximation: M-rank vs Tucker rank for tensors in \eqref{tensor-generation2}}		
\begin{center}
\begin{tabular}
{|c|c|c|c|}\hline
CP-rank  & Tucker rank & $M^+$ &M$^-$ \\ \hline
\multicolumn{4}{|c|}{Dimension $10 \times 10 \times 10 \times 10$}\\ \hline
(*) $r=k=2$, CP-rank  $\le 8$ &   (4,4,4,4) & 8 & 2 \\ \hline
$r=k=3$, CP-rank  $\le 27$ &  (9,9,9,9) & 27 & 3 \\ \hline
$r=k=4$,  CP-rank  $\le 64$ &  (10,10,10,10) & 64 & 4 \\ \hline
\multicolumn{4}{|c|}{Dimension $10 \times 10 \times 15 \times 15$}\\ \hline
(*) $r=k=2$,  CP-rank  $\le 8$ &   (4,4,4,4) & 8 & 2 \\ \hline
$r=k=3$, CP-rank  $\le 27$ &  (9,9,9,9) & 27 & 3 \\ \hline
$r=k=4$,  CP-rank  $\le 64$  & (10,10,15,15) &  64 & 4 \\ \hline
\multicolumn{4}{|c|}{Dimension $15 \times 15 \times 20 \times 20$}\\ \hline
(*) $r=k=2$,  CP-rank  $\le 8$ &   (4,4,4,4) & 8 & 2 \\ \hline
$r=k=3$,  CP-rank  $\le 27$ &  (9,9,9,9) & 27 & 3 \\ \hline
$r=k=4$,  CP-rank  $\le 64$  &(16,16,16,16) &  64 & 4 \\ \hline
\multicolumn{4}{|c|}{Dimension $20 \times 20 \times 20 \times 20$}\\ \hline
$r=k=3$,  CP-rank  $\le 27$ &  (9,9,9,9) & 27 & 3 \\ \hline
$r=k=4$, CP-rank  $\le 64$  &(15,15,16,16) &  64 & 4 \\ \hline
$r=k=5$, CP-rank  $\le 125 $ &  (20,20,20,20) & 125 & 5 \\ \hline
\multicolumn{4}{|c|}{Dimension $20 \times 20 \times 30 \times 30$}\\ \hline
$r=k=3$, CP-rank  $\le 27$ &  (9,9,9,9) & 27 & 3 \\ \hline
$r=k=4$, CP-rank  $\le 64$ &  (16,16,16,16) & 64 & 4 \\ \hline
$r=k=5$, CP-rank  $\le 125 $ &  (20,20,25,25) & 125 & 5 \\ \hline
  \end{tabular}
\end{center}
\label{CP-approximation-tab2}}
\end{table}
From Table \ref{CP-approximation-tab2} we see that $\rank_{M^+}(\mathcal{T}) = rk^2 $ for all instances. This further implies that   the CP-rank of the generated tensors is always equal to $rk^2$ and the M$^+$-rank equals the CP-rank. Moreover, as shown in the rows that are marked by (*) in Table~\ref{CP-approximation-tab2}, the CP-rank is strictly less than any dimension size of the tensor, but the Tucker rank is still strictly less than the CP-rank. This again confirms that the M$^+$-rank is a much better approximation of the CP-rank compared with the averaged Tucker rank.

To summarize, results in both Table~\ref{CP-approximation-tab1} and Table~\ref{CP-approximation-tab2} suggest that $\rank_{M^+}(\mathcal{T}) = \rank_{CP}(\mathcal{T})$ when $\mathcal{T}$ is randomly generated from either~\eqref{tensor-generation} or~\eqref{tensor-generation2}, while there is a substantial gap between the averaged Tucker rank and the CP-rank. Therefore, $\rank_{M^+}$ is a better estimation than the Tucker rank for estimating the CP-rank, at least under the settings considered in our experiments.

\subsection{Synthetic Data for Low-Rank Tensor Optimization Problems}

\subsubsection{Low-M-rank Tensor Completion}

In this subsection we use the FPCA algorithm proposed in \cite{Ma-Goldfarb-Chen-2008} to solve the low-M-rank tensor completion \eqref{matrix-rank-min-nuclear} for fourth-order tensor. {We compare it with the low-n-rank tensor completion \eqref{prob:low-n-rank-min-nuclear}, t-SVD based tensor completion \cite{ZhangAeron2017} and KBR based tensor completion \cite{XieZhaoMengXu2017}. The results are reported in Tables \ref{tensor-completion-tab1} and \ref{tensor-completion-tab}, respectively.} The testing examples are generated as follows. We generate random complex-valued tensor $\mathcal{X}_0$ based on~\eqref{tensor-generation} with various tensor dimension sizes and different values of $r$ so that the CP-rank of the resulting tensor is less than or equal to $r$.
Under each setting, we generate $20$ instances and randomly sample $70\%$, $50\%$, $30\%$ of the entries as the observed ones for tensor completion.
We also use the code in~\cite{Goldfarb-Qin-tensor-2013} to solve the low-n-rank tensor completion problem~\eqref{prob:low-n-rank-min-nuclear}.
We report the average of the Tucker rank, the average of the M$^+$-rank and the average of the M$^-$-rank of the recovered tensor $\mathcal{X}^*$ over the 20 instances in Table \ref{tensor-completion-tab1}, where ``sr'' denotes the sampling ratio, i.e., the percentage of observed entries. We also report the average of the relative errors for both approaches in Table \ref{tensor-completion-tab1}, where the relative error is defined as
\[Err := \frac{\|\mathcal{X}^*-\mathcal{X}_0\|_F}{\|\mathcal{X}_0\|_F}.\]
The CPU times reported are in seconds.

\begin{table}[ht]{\footnotesize
		\caption{Low-M-rank tensor completion vs.\ low-n-rank tensor completion}
\begin{center}
\begin{tabular}
{|c|c|c|c|c|c|}\hline
sr & \multicolumn{2}{|c|}{low-n-rank completion }&\multicolumn{3}{|c|}{low-M-rank completion }\\ \hline
               & Err & Tucker rank & Err & M$^+$ &M$^-$ \\ \hline
\multicolumn{6}{|c|}{Dimension $10 \times 10 \times 10 \times 10$,  \;$r =2$,\; CP-rank $\le 2$ }\\ \hline
70\%&1.23e-5 & (2, 2, 2, 2) & 1.83e-5 & 2&   2 \\
50\%&1.60e-5 & (2, 2, 2, 2) & 1.13e-5 & 2&   2 \\
30\%&7.79e-2 & (10, 10, 10, 10) & 1.18e-4 & 2.67&   2 \\
\hline
\multicolumn{6}{|c|}{Dimension $10 \times 10 \times 10 \times 10$,  \;$r =4$,\; CP-rank $\le 4$ }\\ \hline
70\%&1.80e-5 & (4, 4, 4, 4) & 3.24e-5 & 4 & 4  \\
50\%&2.04e-1 & (10, 10, 10, 10) & 1.78e-5  & 4 & 4 \\
30\%&6.31e-1 & (10, 10, 10, 10) & 2.69e-4  & 5.33 & 4 \\
\hline
\multicolumn{6}{|c|}{Dimension $10 \times 10 \times 10 \times 10$,  \;$r =6$,\; CP-rank $\le 6$ }\\ \hline
70\%&1.69e-1 & (10, 10, 10, 10) & 4.69e-5 & 6& 6 \\
50\%&4.50e-1 & (10, 10, 10, 10) & 2.62e-5 & 6& 6 \\
30\%&7.18e-1 & (10, 10, 10, 10) & 4.97e-6 & 6& 6 \\
\hline
\multicolumn{6}{|c|}{Dimension $15 \times 15 \times 15 \times 15$,  \;$r =3$,\; CP-rank $\le 3$ }\\ \hline
70\%&1.24e-5 & (3, 3, 3, 3) & 6.67e-6 & 3& 3   \\
50\%&1.63e-5 & (3, 3, 3, 3) & 1.83e-6 & 3& 3  \\
30\%&1.57e-3 & (3, 3, 3, 3) & 5.49e-5 & 3& 3   \\ \hline
\multicolumn{6}{|c|}{Dimension $15 \times 15 \times 15 \times 15$,  \;$r =6$,\; CP-rank $\le 6$ }\\ \hline
70\%&1.57e-5 & (6, 6, 6, 6) & 1.22e-5 & 6& 6  \\
50\%&1.73e-1 & (15, 15, 15, 15) & 3.90e-6 & 6& 6  \\
30\%&6.10e-1 & (15, 15, 15, 15) & 7.78e-5 & 6& 6   \\ \hline
\multicolumn{6}{|c|}{Dimension $15 \times 15 \times 15 \times 15$,  \;$r =9$,\; CP-rank $\le 9$ }\\ \hline
70\%&1.56e-1 & (15, 15, 15, 15) & 2.13e-5 & 9& 9  \\
50\%&4.47e-1 & (15, 15, 15, 15) & 6.80e-6 & 9& 9   \\
30\%&7.16e-1 & (15, 15, 15, 15) & 1.40e-4 & 9& 9  \\ \hline
\multicolumn{6}{|c|}{Dimension $20 \times 20 \times 20 \times 20$,  \;$r =4$,\; CP-rank $\le 4$ }\\ \hline
70\%&7.26e-6 & (4, 4, 4, 4) & 3.49e-6 & 4& 4   \\
50\%&1.45e-5 & (4, 4, 4, 4) & 6.76e-7 & 4& 4   \\
30\%&1.96e-5 & (4, 4, 4, 4) & 2.37e-5 & 4& 4   \\ \hline
\multicolumn{6}{|c|}{Dimension $20 \times 20 \times 20 \times 20$,  \;$r =8$,\; CP-rank $\le 8$ }\\ \hline
70\%&1.67e-5 & (8, 8, 8, 8) & 6.60e-6 & 8& 8  \\
50\%&1.34e-1 & (20, 20, 20, 20) & 1.59e-6 & 8& 8  \\
30\%&5.87e-1 & (20, 20, 20, 20) & 4.24e-5 & 8& 8   \\ \hline
\multicolumn{6}{|c|}{Dimension $20 \times 20 \times 20 \times 20$,  \;$r =12$,\; CP-rank $\le 12$ }\\ \hline
70\%&1.45e-1 & (20, 20, 20, 20) & 8.58e-6 & 12& 12   \\
50\%&4.20e-1 & (20, 20, 20, 20) & 3.46e-6 & 12& 12    \\
30\%&6.99e-1 & (20, 20, 20, 20) & 6.75e-5 & 12& 12     \\ \hline
\end{tabular}
\end{center}
\label{tensor-completion-tab1}}
\end{table}

From Table \ref{tensor-completion-tab1} we see that when the CP-rank is not very small, the low-n-rank completion approach fails to recover the original tensor while our low-M-rank method works well and its relative error is usually at the order of $10^{-5}$ or $10^{-6}$.
 This result is not surprising. To illustrate this, let us take a $10 \times 10 \times 10 \times 10$ tensor as an example.
 In this case the mode-$n$ unfolding and the square unfolding will result in a $10 \times 1000$ and a $100 \times 100$ matrices respectively.
 When the CP-rank of the underlying tensor is equal to $6$, it is a relatively high rank for a $10 \times 1000$ matrix, while it is a relatively low rank for a $100 \times 100$ matrix.
This is exactly what happens in the third row-block of Table~\ref{tensor-completion-tab1} when the dimension of the tensor is $10\times 10\times 10\times 10$ and the CP-rank is $6$. In addition, we note that the Tucker rank is often larger than the CP-rank when it fails to complete the original tensor. However, the M$^+$-rank and M$^-$-rank are almost always equal to the CP-rank except for only two cases. This again suggests that the M-rank is a good approximation to the CP-rank. 
Moreover, the results in Table \ref{tensor-completion-tab1} also show that the low-M-rank tensor completion model has a much better recoverability than the low-n-rank model in terms of the relative error of the recovered tensors.
{Furthermore, we conduct similar tests on our low-M-rank model \eqref{matrix-rank-min-nuclear}, low-n-rank model, t-SVD based tensor completion and KBR based tensor completion} when the CP-rank is larger than the size of one dimension of the tensor. This time we assume only $30\%$ of entries are observed. {The averaged errors and run time in seconds over $20$ randomly generated instances for different tensor dimension sizes are summarized in Table~\ref{tensor-completion-tab}, which show that our low-M-rank tensor completion model can still recover the tensors very well while the other three methods struggled.}

In another set of tests, we aim to observe the relationship between the M-rank and the symmetric CP-rank via solving the super-symmetric tensor completion problem~\eqref{matrix-rank-min-nuclear-sym}. To this end, we randomly generate $20$ complex-valued tensors $\mathcal{X}_0$ in the form~\eqref{tensor-generation} for different choice of tensor dimensions and values of $r$, so that the symmetric CP-rank of the resulting tensor is less than or equal to $r$.
For each generated tensor, we randomly remove $60\%$ of the entries and then solve the tensor completion problem~\eqref{matrix-rank-min-nuclear-sym}. The results are reported in Table~\ref{sym-tensor-completion-tab}, which suggests that the original tensor is nicely recovered (all with the relative error at the order of $10^{-6}$). Moreover, the M-rank and the symmetric CP-rank shown in the table are {\it always} identical, which implies that the M-rank remains a good approximation to the symmetric CP-rank for super-symmetric tensor.
We also note that solving problem~\eqref{matrix-rank-min-nuclear-sym} is much more time consuming than solving~\eqref{matrix-rank-min-nuclear}, due to the super-symmetric constraint which is essentially equivalent to $O(n^{4})$ linear constraints and is costly to deal with.


\begin{table*}[t]{\footnotesize
\caption{Comparison of $4$ tensor completion methods for randomly generated instances}
\begin{center}
\begin{tabular}
{|c|c|c|c|c|c|c|c|c|}\hline
    & \multicolumn{2}{|c|}{low-n-rank} & \multicolumn{2}{|c|}{low-M-rank} & \multicolumn{2}{|c|}{KBR} & \multicolumn{2}{|c|}{t-SVD} \\\hline
$r$ & $Err$ & CPU & $Err$ & CPU & $Err$ & CPU & $Err$ & CPU \\ \hline
\multicolumn{9}{|c|}{Dimension $10 \times 10 \times 10 \times 10$}\\ \hline
12 & 7.77e-01 & 3.70 & 6.92e-05 & 2.50 & 6.62e-02 & 4.29 & 1.11e-01 & 6.56 \\ \hline
\multicolumn{9}{|c|}{Dimension $10 \times 10 \times 13 \times 13$}\\\hline
15 & 7.80e-01 & 5.41 & 1.62e-04 & 4.39 & 6.27e-02 & 5.38 & 1.06e-01 & 8.45 \\ \hline
\multicolumn{9}{|c|}{Dimension $20 \times 20 \times 20 \times 20$}\\\hline
24 & 7.81e-01 & 66.33 & 3.66e-06 & 29.15 & 3.90e-02 & 42.75 & 8.79e-02 & 31.57  \\ \hline
\multicolumn{9}{|c|}{Dimension $20 \times 20 \times 25 \times 25$}\\\hline
28 & 7.83e-01 & 120.55 & 9.83e-04 & 1.21 & 3.56e-02 & 65.63 & 7.96e-02 & 46.48   \\ \hline
\multicolumn{9}{|c|}{Dimension $30 \times 30 \times 30 \times 30$}\\\hline
35 &  7.82e-01 & 999.19 & 8.92e-05 & 4.42 & 2.78e-02 & 295.44 & 7.44e-02 & 130.50  \\ \hline
\multicolumn{9}{|c|}{Dimension $35 \times 35 \times 40 \times 40$}\\\hline
45 & 7.75e-01 & 1624.32 & 5.91e-05 & 7.43 & 2.83e-02 & 539.94 & 7.64e-02 & 176.83  \\ \hline
\end{tabular}
\end{center}
\label{tensor-completion-tab}}
\end{table*}

\begin{table}[ht]{\footnotesize
		\caption{Numerical results for low-M-rank super-symmetric tensor completion by solving \eqref{matrix-rank-min-nuclear-sym}}
\begin{center}
\begin{tabular}
{|c|c|c|c|}\hline
 $\rank_{SCP}(\mathcal{X}_0)$ & $Err$ & CPU & $\rank_{M}(\mathcal{X}^*)$ \\ \hline
\multicolumn{4}{|c|}{Dimension $10 \times 10 \times 10 \times 10$}\\ \hline
$r=8,\; \rank_{SCP}(\mathcal{X}_0)\le 8$ & 8.83e-006 & 5.96 & 8 \\ \hline
$r=12,\; \rank_{SCP}(\mathcal{X}_0)\le 12$ & 6.81e-006 & 8.33 & 12 \\
 \hline
  \multicolumn{4}{|c|}{Dimension $15 \times 15 \times 15 \times 15$}\\\hline
$r=8,\; \rank_{SCP}(\mathcal{X}_0)\le 8$  & 8.95e-006 & 64.66 & 8 \\ \hline
$r=20,\; \rank_{SCP}(\mathcal{X}_0)\le 20$  & 6.19e-006 & 89.40 & 20 \\
\hline
\multicolumn{4}{|c|}{Dimension $20 \times 20 \times 20 \times 20$}\\\hline
$r=15,\; \rank_{SCP}(\mathcal{X}_0)\le 15$  & 6.63e-006 & 523.54 & 15 \\ \hline
$r=25,\; \rank_{SCP}(\mathcal{X}_0)\le 25$  & 6.46e-006 & 567.30 & 25 \\
\hline
\multicolumn{4}{|c|}{Dimension $25 \times 25 \times 25 \times 25$}\\\hline
$r=15,\; \rank_{SCP}(\mathcal{X}_0)\le 15$  & 4.21e-006 & 1109.26& 15 \\ \hline
$r=30,\; \rank_{SCP}(\mathcal{X}_0)\le 30$  & 3.87e-006 & 2470.67& 30\\
\hline
  \end{tabular}
\end{center}
\label{sym-tensor-completion-tab}}
\end{table}

\subsubsection{Low-M-Rank Robust Tensor PCA}

In this subsection we report the numerical results for the robust tensor PCA problem \eqref{matrix-rpca-nuclear} {based on the low-M-rank, the low-n-rank, the t-SVD based rank and the KBR based rank models.} We choose $\lambda = 1/\sqrt{n_1n_2}$ in \eqref{matrix-rpca-nuclear}, and apply the alternating linearization method in \cite{Aybat-Goldfarb-Ma-RPCA-2012} to solve \eqref{matrix-rpca-nuclear}.
The low-rank tensor $\mathcal{Y}_0$ is randomly generated according to formula~\eqref{tensor-generation}, and a random complex-valued sparse tensor $\mathcal{Z}_0$ is generated with cardinality of $0.05\cdot n_1 n_2 n_3 n_4$. Then we set $\mathcal{F} = \mathcal{Y}_0 + \mathcal{Z}_0 $ as the observed data. We conduct our tests under various settings of tensor dimension sizes and CP-ranks of $\mathcal{Y}_0$. For each setting, $20$ instances are randomly generated for test. We
also apply the code in~\cite{Goldfarb-Qin-tensor-2013} to solve the low-n-rank robust tensor PCA problem:
\begin{equation}\label{prob:low-n-rank-RPCA}
\min\limits_{\mathcal{Y},\mathcal{Z}\in\C^{n_1\times n_2\cdots \times n_{d}}} \ \frac{1}{d}\sum_{j=1}^d \|Y(j)\|_*+ \lambda \|\mathcal{Z}\|_1, \quad \st \ \quad \mathcal{Y} + \mathcal{Z} = \mathcal{F}.
\end{equation}
{The details of the t-SVD based and the KBR based robust tensor PCA can be found in \cite{LuEtal2016, XieZhaoMengXu2017}.

Suppose $\mathcal{Y}^*$ and $\mathcal{Z}^*$ are the low-rank tensor and sparse tensor returned by the algorithms. We define the relative error of the low-rank part and sparse part as
\[Err_{L} := \frac{\|\mathcal{Y}^* -\mathcal{Y}_0\|_F}{\|\mathcal{Y}_0 \|_F}, \qquad Err_{S} := \frac{\|\mathcal{Z}^* -\mathcal{Z}_0\|_F}{\|\mathcal{Z}_0 \|_F}.\]
We report the average of these two relative errors, the recovered low-rank tensor $\mathcal{Y}^*$, and the average run time (in seconds) of each method over $20$ instances in Table \ref{tensor-rpca-tab}.
Results in Table \ref{tensor-rpca-tab} suggest that in many cases, the low-n-rank, t-SVD based rank and the KBR based rank robust tensor PCA models fail to extract the low-rank part while our low-M-rank robust tensor PCA model can always recover the tensor with relative error at the order of $10^{-3}$ to $10^{-6}$. }

\begin{table*}[t]
	{\scriptsize
		\caption{Comparison of $4$ robust tensor PCA methods for randomly generated instances  }
		\begin{center}
			\begin{tabular}
				{|c|c|c|c|c|c|c|c|c|c|c|c|c|}\hline
				&\multicolumn{3}{|c|}{low-n-rank RPCA}&\multicolumn{3}{|c|}{low-M-rank RPCA}&\multicolumn{3}{|c|}{KBR RPCA}&\multicolumn{3}{|c|}{t-SVD RPCA}\\ \hline
				CP & $\mbox{Err}_{{S}}$ & $\mbox{Err}_{{L}}$ & CPU & $\mbox{Err}_{S}$  & $\mbox{Err}_{L}$ & CPU & $\mbox{Err}_{S}$  & $\mbox{Err}_{L}$ &  CPU & $\mbox{Err}_{S}$  & $\mbox{Err}_{L}$ &  CPU  \\ \hline
				\multicolumn{13}{|c|}{Dimension $10 \times 10 \times 10 \times 10$}\\ \hline
				$\le 2$ &8.08e-01 & 1.03e-01 & 31.10 & 2.09e-02 & 5.99e-04 & 0.09 & 7.09e-01 & 1.61e-01 & 0.29 & 7.24e-01 & 1.57e-01 & 2.52 \\ \hline
				$\le 4$ & 8.98e-01 & 1.18e-01 & 30.00 & 5.08e-02 & 1.46e-03 & 0.08 & 7.18e-01 & 1.67e-01 & 0.31 & 7.87e-01 & 1.66e-01 & 2.50 \\ \hline
				$\le 6$ & 1.06e+00 & 1.36e-01 & 33.69 & 6.52e-02 & 1.80e-03 & 0.09 & 8.39e-01 & 1.73e-01 & 0.46 & 9.17e-01 & 1.72e-01 & 2.71  \\ \hline
				$\le 8$ & 1.08e+00 & 1.43e-01 & 30.98 & 8.26e-02 & 2.42e-03 & 0.09 & 8.57e-01 & 1.82e-01 & 0.42 & 9.47e-01 & 1.82e-01 & 2.61   \\ \hline
				$\le 12$ & 1.17e+00 & 1.45e-01 & 34.39 & 1.53e-01 & 4.26e-03 & 0.09 & 9.22e-01 & 1.76e-01 & 0.50 & 1.02e+00 & 1.76e-01 & 2.88   \\ \hline
				\multicolumn{13}{|c|}{Dimension $15 \times 15 \times 15 \times 15$}\\ \hline
				$\le 3$ &2.18e-01 & 2.94e-02 & 53.58 & 4.84e-03 & 8.68e-05 & 0.52 & 5.24e-01 & 1.57e-01 & 0.26 & 4.59e-01 & 1.48e-01 & 9.20  \\ \hline			
				$\le 6$& 3.66e-01 & 4.80e-02 & 55.83 & 8.33e-03 & 1.50e-04 & 0.52 & 5.84e-01 & 1.57e-01 & 0.36 & 5.79e-01 & 1.50e-01 & 9.32   \\ \hline		
				$\le 9$& 4.75e-01 & 6.05e-02 & 54.71 & 1.65e-02 & 2.97e-04 & 0.50 & 6.29e-01 & 1.57e-01 & 0.41 & 6.50e-01 & 1.51e-01 & 8.61    \\ \hline		
				$\le 12$& 5.77e-01 & 7.28e-02 & 58.53 & 1.65e-02 & 2.92e-04 & 0.51 & 7.18e-01 & 1.62e-01 & 0.52 & 7.49e-01 & 1.57e-01 & 8.97     \\ \hline	
				$\le 18$& 6.39e-01 & 8.48e-02 & 55.35 & 2.08e-02 & 3.90e-04 & 0.49 & 7.33e-01 & 1.72e-01 & 0.57 & 7.86e-01 & 1.68e-01 & 9.00      \\ \hline														
				\multicolumn{13}{|c|}{Dimension $20 \times 20 \times 20 \times 20$}\\ \hline
				$\le 4$ &7.42e-02 & 9.86e-03 & 186.26 & 4.19e-04 & 6.39e-06 & 2.23 & 4.17e-01 & 1.47e-01 & 0.54 & 2.87e-01 & 1.37e-01 & 30.39   \\ \hline	
				$\le 8$ & 1.42e-01 & 1.84e-02 & 181.94 & 4.24e-03 & 5.47e-05 & 2.16 & 5.16e-01 & 1.51e-01 & 0.77 & 4.10e-01 & 1.38e-01 & 28.97   \\ \hline	
				$\le 12$ & 2.19e-01 & 2.81e-02 & 207.48 & 4.51e-03 & 6.02e-05 & 2.16 & 5.85e-01 & 1.54e-01 & 0.97 & 5.02e-01 & 1.41e-01 & 22.49   \\ \hline	
				$\le 16$ & 2.74e-01 & 3.53e-02 & 203.18 & 5.85e-03 & 8.22e-05 & 2.08 & 6.22e-01 & 1.58e-01 & 1.03 & 5.57e-01 & 1.46e-01 & 20.91    \\ \hline	
				$\le 24$ & 3.36e-01 & 4.46e-02 & 198.84 & 7.17e-03 & 9.86e-05 & 1.98 & 7.03e-01 & 1.70e-01 & 1.21 & 6.62e-01 & 1.57e-01 & 20.80  \\															 \hline
			\end{tabular}
		\end{center}
		\label{tensor-rpca-tab}}
\end{table*}

\subsection{Colored Video Completion and Decomposition}

In this subsection, we apply the low-M-rank tensor completion and low-M-rank robust tensor PCA to colored video completion and decomposition. A colored video consists of $n_4$ frames, and each frame is an image stored in the RGB format as an $n_1\times n_2\times n_3$  array. As a result, filling in the missing entries of the colored video and decomposing the video to static background and moving foreground can be regarded as low-rank tensor completion \eqref{prob:tensor-completion} and robust tensor PCA \eqref{prob:robust-tensor-recovery}, respectively.

In our experiment for tensor completion, we chose $50$ frames from a video taken in a lobby, which was introduced by Li \etal in \cite{Li-Huang-Gu-Tian-2004}. Each frame in this video is a colored image with size $128\times 160\times 3$. The images in the first row of Figure \ref{CompletionPics70Missing} are three frames of the video. Basically we chose the $50$ frames such that they only contain static background, and thus the $128\times 160\times 3 \times 50$ fourth-order tensor formed by them are expected to have low rank, because the background is almost the same in each frame. We then randomly remove 70\% of the entries from the video, and the images in the second row of Figure \ref{CompletionPics70Missing} are the frames after the removal. We then apply the FPCA algorithm \cite{Ma-Goldfarb-Chen-2008} to solve \eqref{matrix-rank-min-nuclear} with the square unfolding matrix with matrix size $20480 \times 150$ and with certain single mode-n unfolding matrix respectively, to complete the missing entries in the target tensor. {In addition, the t-SVD based tensor completion \cite{ZhangAeron2017} and KBR based tensor completion \cite{XieZhaoMengXu2017} are also performed, and the algorithm in \cite{Goldfarb-Qin-tensor-2013} is applied to solve low-n-rank minimization model \eqref{prob:low-n-rank-min} for comparison. The images in the third to seventh row of Figure \ref{CompletionPics70Missing} are the frames recovered using the five models/methods respectively. We can see that when 70\% tensor entries are missing the recovery capabilities of square unfolding,  single mode-n unfolding, t-SVD based completion and KBR based completion are comparable, which are all slightly better than that of low-n-rank minimization model. This can be seen from the blurred red flower in the pot pointed by the red arrow in the fifth row of Figure \ref{CompletionPics70Missing}. Next, we further increase the portion of the missing data to 80\% and the recovered videos are shown in Figure \ref{CompletionPics80Missing}. We see that the third row does not retain the correct color and the fifth row and the sixth row are blurred, but the images in the fourth row and seventh row are almost as good as the original images in the first row. This implies that square unfolding is not necessarily always better than mode-n unfolding for this kind of problem, which further implies that the video data may not have low CP-rank. Minimizing the average nuclear norm of all mode-n unfolding, which is popular in tensor completion \cite{LMWY09, Gandy-Recht-Yamada-2011, Goldfarb-Qin-tensor-2013}, however, is always worse than an appropriate single mode-n minimization. Determining the most appropriate mode to unfold the tensor is highly dependent on the structure and the physical meaning of the underlying tensor model. In conclusion, we observe that for each specific application, it pays to learn the correct mode of matricization, square or otherwise.}

In our experiment for robust tensor PCA, we chose another $50$ frames from the same video in \cite{Li-Huang-Gu-Tian-2004}. These frames were chosen such that the frames contain some moving foregrounds. The task in robust tensor PCA is to decompose the given tensor into two parts: a low-rank tensor corresponding to the static background, and a sparse tensor corresponding to the moving foreground. Note that the tensor corresponding to the moving foreground is sparse because the foreground usually only occupies a small portion of the frame.
Thus this decomposition can be found by solving the robust tensor PCA problem \eqref{prob:robust-tensor-recovery}. Here we again apply the alternating linearization method proposed in \cite{Aybat-Goldfarb-Ma-RPCA-2012} to solve \eqref{matrix-rpca-nuclear} for the task of robust tensor PCA, where $\lambda$ in~\eqref{matrix-rpca-nuclear} is chosen as $1/\sqrt{n_1 n_2}$ and $n_1, n_2$ are the first two dimensions of the fourth-order tensor.
The decomposition results are shown in Figure \ref{graph:robust-recovery}. The images in the first row of Figure \ref{graph:robust-recovery} are frames of the original video. The images in the second and third rows of Figure \ref{graph:robust-recovery} are the corresponding static background and moving foreground, respectively. We can see that our low-M-rank robust tensor PCA approach very effectively decomposes the video, which is a fourth-order tensor.


\begin{figure}
	\begin{center}
		\includegraphics[scale=0.55]{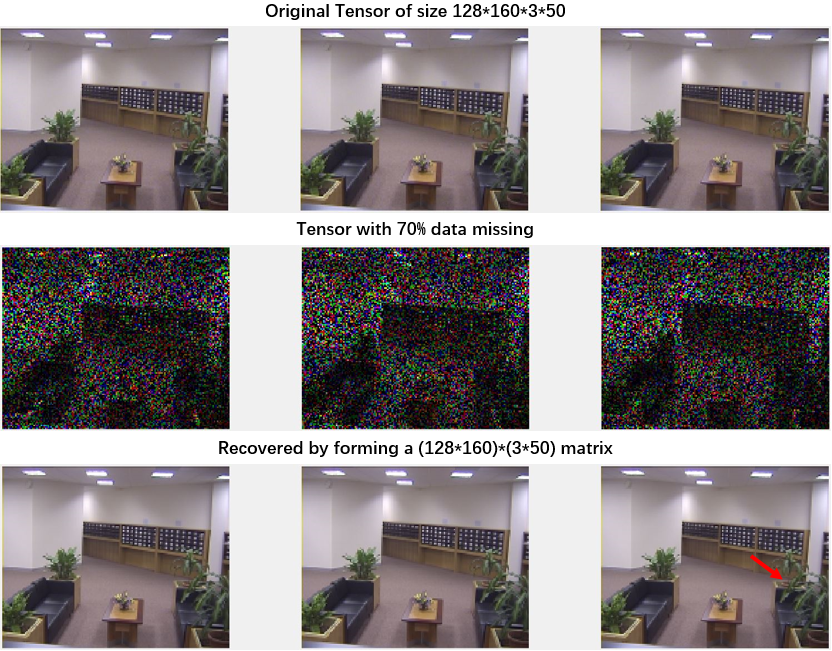}
        \includegraphics[scale=0.55]{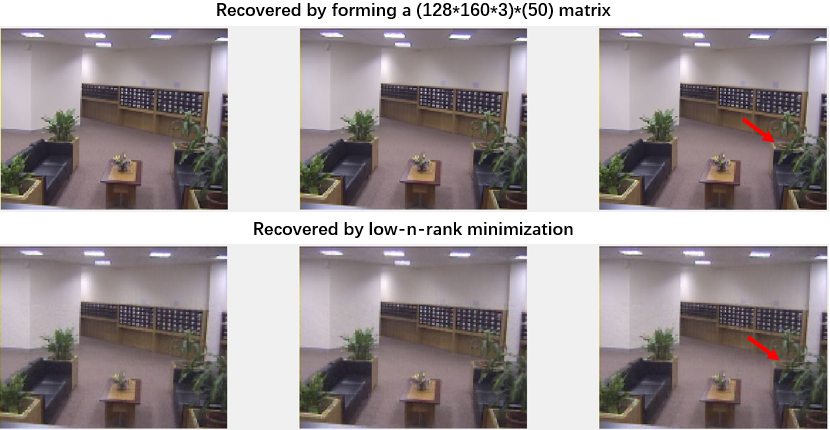}
        \includegraphics[scale=0.55]{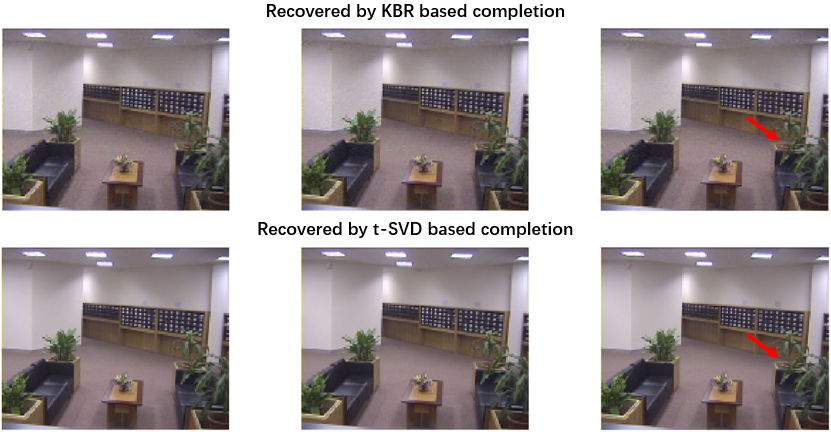}
		\caption{Colored video completion with 70\% missing entries}
\label{CompletionPics70Missing}
	\end{center}
\end{figure}

\begin{figure}
	\begin{center}
		\includegraphics[scale=0.55]{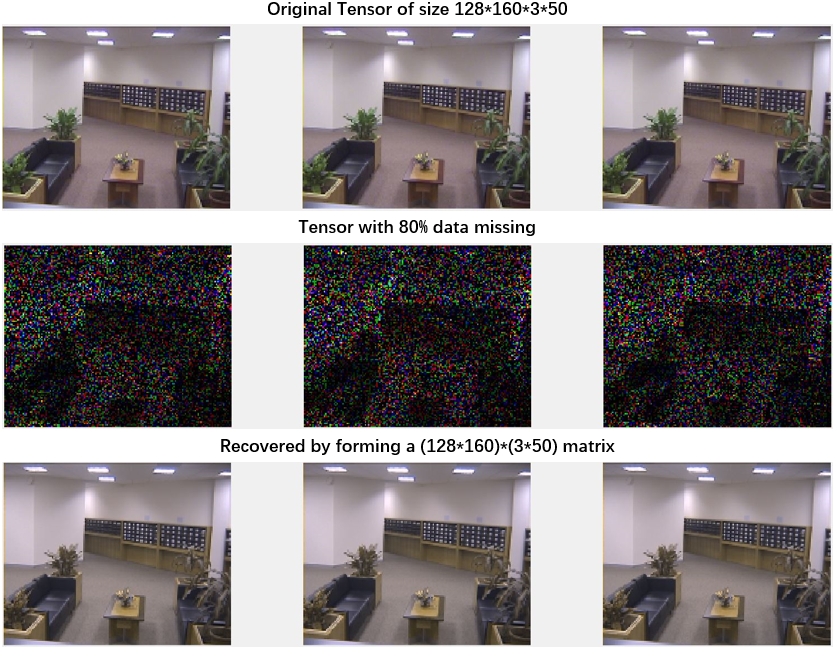}
        \includegraphics[scale=0.55]{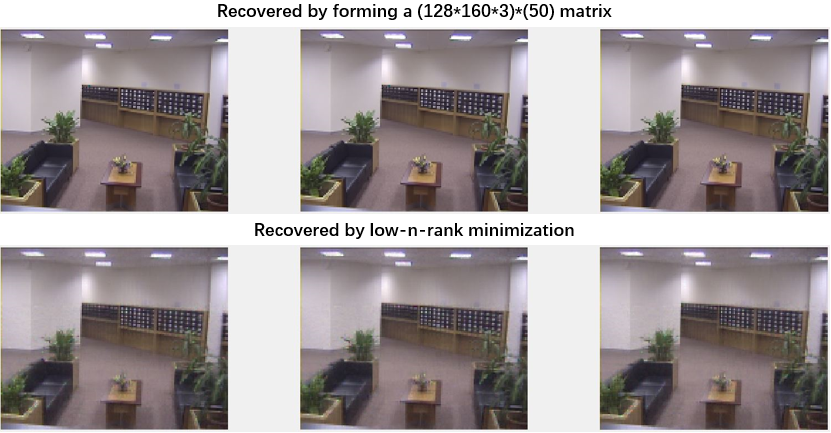}
        \includegraphics[scale=0.55]{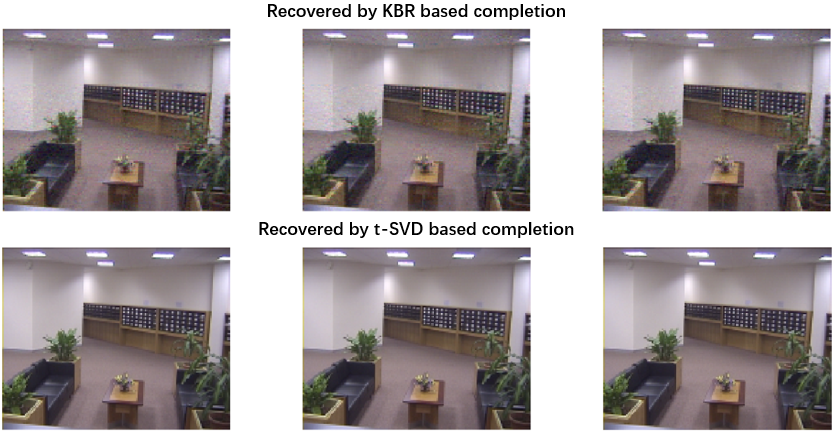}
		\caption{Colored video completion with 80\% missing entries}
		\label{CompletionPics80Missing}
	\end{center}
\end{figure}

\begin{figure}
\begin{center}
  \includegraphics[scale=0.4]{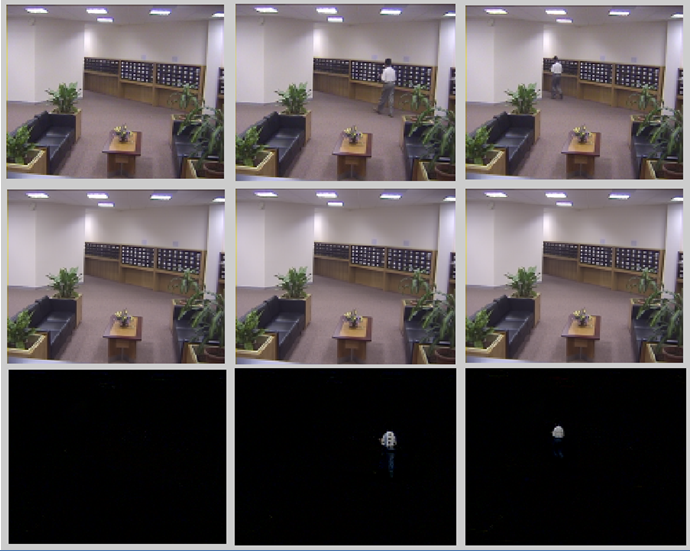}\\
  \caption{Robust video recovery as robust tensor PCA. The first row are the $3$ frames of the original video sequence. The second row are the recovered background, and the last row are the recovered foreground.}\label{graph:robust-recovery}
\end{center}
\end{figure}


\section{Concluding Remarks}

{In this paper, we proposed a new approach for solving tensor completion and robust tensor PCA problems. The approach is based on the newly defined notion of M-rank of a tensor. We provided a theoretical foundation for using the M-rank as a proximity of the CP-rank. Computationally, the M-rank is much easier to compute than the CP-rank. We then solved the low-M-rank tensor completion and robust tensor PCA problems by converting the tensor problems to matrix problems using the M-decomposition. The numerical results show that our method can recover the tensors very well, confirming that the M-rank is a good approximation of the CP-rank in such applications. Compared to the classical mode-$n$ rank, the new M-rank differs in its principle to unfold the tensor into matrices: it takes a balanced approach. Namely, for a $2m$-order tensor, the M-rank unfolding groups the indices by $m\times m$, while the Tucker rank folding groups the indices in the fashion of $1\times (2m-1)$. It is certainly possible to attempt all groupings such as $k\times (2m-k)$ with $k=1,2,...,m$, though the computational costs increase exponentially with the order of the tensor. A balanced folding can also be extended to odd-order tensors; for an $(2m+1)$-order tensor, leading to grouping the indices by $m \times (m+1)$. For the tensors of order 3, this reduces to the traditional mode-$n$ matricization.
Most results in this paper, including CP-rank approximation and low-M-rank tensor recovery, can be easily generalized to odd-order tensors.
}

\section*{Acknowledgments}

We would like to thank Cun Mu for discussions on the implementation of low-n-rank tensor optimization problems. We are also grateful to three anonymous reviewers for constructive suggestions that greatly improved the presentation of the paper.

\bibliographystyle{IEEEtran}

\end{document}